\documentclass[10pt,leqno]{article} 

   \usepackage[centertags]{amsmath}
   \usepackage{amsfonts}
   \usepackage{amsmath}
   \usepackage{amssymb}
   \usepackage{amsthm}
   \usepackage{newlfont}
   \usepackage[latin1]{inputenc}
   \usepackage{graphicx}

 \usepackage{multirow, bigdelim}

\allowdisplaybreaks[3]

\theoremstyle{plain}
\newtheorem{theorem}{Theorem}[section]
\newtheorem{lemma}[theorem]{Lemma}

\newtheorem{corollary}[theorem]{Corollary}

\theoremstyle{definition}
\newtheorem{definition}[theorem]{Definition}

\theoremstyle{remark}
\newtheorem{remark}[theorem]{Remark}
\newtheorem*{remark*}{Remark}

\numberwithin{equation}{section}

\newcommand{\dosfilas}[2]{
  \ldelim[{2}{2mm}& #1 &\rdelim]{2}{2mm} \\
  & #2 & &  & &
}

\newcommand*\pFqskip{8mu}
\catcode`,\active
\newcommand*\pFq{\begingroup
        \catcode`\,\active
        \def ,{\mskip\pFqskip\relax}%
        \dopFq
}
\catcode`\,12
\def\dopFq#1#2#3#4#5{%
        {}_{#1}F_{#2}\biggl(\genfrac..{0pt}{}{#3}{#4};#5\biggr)%
        \endgroup
}

\newcommand\D{{\mathcal D}}

\newcommand\F{{\mathcal F}}
\newcommand\G{{\mathcal G}}
\newcommand\HH{{\mathcal A}}

\newcommand\CC{{\mathbb C}}

\newcommand\RR{{\mathbb R}}
\newcommand\ZZ{{\mathbb Z}}
\newcommand\NN{{\mathbb N}}

\newcommand\X{{\Theta}}
\newcommand\x{{\theta}}

\newcommand\pp{\mbox{$\mathfrak{p}_{F}$}}

\newcommand\rank{\operatorname{rank}}

\newcommand\sign{\operatorname{sign}}
\newcommand\diagonal{\operatorname{diag}}

\newcommand\Sh{\mbox{\Large $\mathfrak {s}$}}

   \title{Exceptional Hahn and Jacobi orthogonal polynomials
  \footnote{Partially supported by MTM2012-36732-C03-03 (Ministerio de Economía y Competitividad),
FQM-262, FQM-4643, FQM-7276 (Junta de Andalucía) and Feder Funds (European
Union).}}
   \author{Antonio J. Dur\'{a}n\\
     \footnotesize
        \  Departamento de An\'{a}lisis Matem\'{a}tico.
       Universidad de Sevilla \\
       \footnotesize Apdo (P. O. BOX) 1160. 41080 Sevilla. Spain.
   duran@us.es \\
          \ \ }
   \date{}
   
\begin{document}
   \maketitle

\bigskip

\begin{abstract}
Using Casorati determinants of Hahn polynomials $(h_n^{\alpha,\beta,N})_n$, we construct for each pair $\F=(F_1,F_2)$ of finite sets of positive integers polynomials $h_n^{\alpha,\beta,N;\F}$, $n\in \sigma _\F$, which are eigenfunctions of a second order difference operator, where $\sigma _\F$ is certain set of nonnegative integers, $\sigma _\F \varsubsetneq \NN$.  When $N\in \NN$ and $\alpha$, $\beta$, $N$ and $\F$ satisfy a suitable admissibility condition, we prove that the polynomials $h_n^{\alpha,\beta,N;\F}$  are also orthogonal and complete with respect to a positive measure (exceptional Hahn polynomials). By passing to the limit, we transform the Casorati determinant of Hahn polynomials into a Wronskian type determinant of Jacobi polynomials $(P_n^{\alpha,\beta})_n$. Under suitable conditions for $\alpha$, $\beta$ and $\F$, these Wronskian type determinants turn out to be exceptional Jacobi polynomials.
\end{abstract}

\section{Introduction}
In \cite{duch} and \cite{dume}, we have introduced a systematic way of constructing exceptional discrete orthogonal polynomials using the concept of dual families of polynomials. We applied  this procedure to construct exceptional Charlier and Meixner polynomials and, passing to the limit, exceptional Hermite and Laguerre polynomials, respectively. The purpose of this paper is to extend this construction to Hahn and Jacobi exceptional polynomials.

Exceptional orthogonal polynomials $p_n$, $n\in X\varsubsetneq \NN$, are complete orthogonal polynomial systems with respect to a positive measure which in addition
are eigenfunctions of a second order differential operator. They extend the  classical families of Hermite, Laguerre and Jacobi. The last few years have seen a great deal of activity in the area  of exceptional orthogonal polynomials (see, for instance,
\cite{duch,dume,GUKM1,GUKM2} (where the adjective \textrm{exceptional} for this topic was introduced), \cite{GUKM3,G,OS0,OS4,Qu,STZ},  and the references therein).

In the same way, exceptional discrete orthogonal polynomials are complete orthogonal polynomial systems with respect to a positive measure which in addition are eigenfunctions of a second order difference operator, extending the discrete classical families of Charlier, Meixner, Krawtchouk and Hahn, or Wilson, Racah, etc., if orthogonal discrete polynomials on nonuniform lattices are considered (\cite{duch,dume,OS6,YZ}).

The most apparent difference between classical or classical discrete orthogonal polynomials and their exceptional counterparts
is that the exceptional families have gaps in their degrees, in the
sense that not all degrees are present in the sequence of polynomials (as it happens with the classical families) although they are complete in the underlying $L^2$ space. This
means that they are not covered by the hypotheses of Bochner's and Lancaster's classification theorems (see \cite{B} or \cite{La}) for classical and classical discrete orthogonal polynomials, respectively.

As mentioned above, we use the concept of dual families of polynomials to construct exceptional discrete orthogonal polynomials (see \cite{Leo}). Classical discrete orthogonal polynomials $(p_n)_n$ are eigenfunctions of two second order difference operators. One of them acts on the discrete parameter $n$, and corresponds with the three term recurrence relation they satisfy as a consequence of their orthogonality with respect to a measure. The other second order difference operator acts on the continuous variable $x$. Duality is with respect to the corresponding sequences of eigenvalues. For Charlier or Meixner polynomials, both sequences can be taken equal to $n$, and the families turns out to be selfdual. For Hahn polynomials the situation is different because one of the eigenvalue sequences is quadratic in $n$. As a consequence, when duality is applied one moves from Hahn to dual Hahn polynomials and viceversa.

\begin{definition}\label{dfp}
Given two sequences of numbers $(\varpi_n)_{n\in V}$ and $(\varpi ^*_n)_{n\in U}$, where $U,V$ are subsets of $\NN$,
we say that the two sequences of polynomials $(p_n)_{n\in U}$, $(q_n)_{n\in V}$ are dual with respect to $(\varpi_n)_{n\in V}$ and $(\varpi ^*_n)_{n\in U}$ if there exist a couple of sequences of numbers $(\xi_n)_{n\in U}, (\zeta_n)_{n\in V} $ such that
\begin{equation}\label{defdp}
\xi_up_u(\varpi _v)=\zeta_vq_v(\varpi ^*_u), \quad u\in U, v\in V.
\end{equation}
\end{definition}

Duality has shown to be a fruitful concept regarding discrete orthogonal polynomials, and his utility has been again manifest in the exceptional discrete polynomials world. Indeed, as we pointed out in \cite{duch} and \cite{dume}, it turns out that duality interchanges exceptional discrete orthogonal polynomials with the so-called Krall discrete orthogonal polynomials. A Krall discrete orthogonal family is a sequence of polynomials $(p_n)_{n\in \NN}$, $p_n$ of degree $n$, orthogonal with respect to a positive measure which, in addition, are also eigenfunctions of a higher order difference operator. A huge amount of families of Krall discrete orthogonal polynomials have been recently introduced by the author in \cite{dudh} by mean of certain Christoffel transform of the classical discrete measures of dual Hahn (see also \cite{du0,du1,DdI,DdI2}). A Christoffel transform  consists in multiplying a measure $\mu$ by a polynomial $r$.

The content of this paper is as follows. In Section 2, we include some preliminary results about Christoffel transforms and finite sets of positive integers.

In Section 3, using Casorati determinants of Hahn polynomials we associated to a pair $\F=(F_1,F_2)$ of finite sets of positive integers a sequence of polynomials which are eigenfunctions of a second order difference operator.
Denote by  $k_i$ for the number of elements of $F_i$,
$i=1,2$ and by $k=k_1+k_2$ the number of elements of $\F$. One of
the components of $\F$, but not both, can be the empty set.
We define the nonnegative integer $u_\F$  by $u_\F=\sum_{f\in F_1}f+\sum_{f\in
F_2}f-\binom{k_1+1}{2}-\binom{k_2}{2}$ and the infinite set of nonnegative integers $\sigma _\F$ by
$$
\sigma _\F=\{u_\F,u_\F+1,u_\F+2,\cdots \}\setminus \{u_\F+f,f\in F_1\}.
$$
Under mild conditions on the parameters $\alpha, \beta $ and $N$, we then associate to the pair $\F$ the sequence of polynomials $h_n^{\alpha,\beta,N;\F}$, $n\in \sigma _\F$, defined by
\begin{equation}\label{defmexi}
h_n^{\alpha, \beta, N;\F}(x)= \frac{ \left|
  \begin{array}{@{}c@{}lccc@{}c@{}}
    & h_{n-u_\F}^{\alpha ,\beta ,N}(x+j-1) &&\hspace{-.9cm}{}_{1\le j\le k+1} \\
    \dosfilas{ h_{f}^{\alpha ,\beta ,N}(x+j-1) }{f\in F_1} \\
    \dosfilas{ s_{k-j+1}^{N-k+1}(x)s_{j-1}^{N+\beta+1}(x+j-1)h_{f}^{\alpha ,-\beta ,\beta+N}(x+j-1) }{f\in F_2}
  \end{array}
  \hspace{-.4cm}\right|}{\prod_{i=1}^{k_2}(N-x-k+1)_{k_2-i}(N+\beta-x-i+2)_{i-1}}
\end{equation}
where $(h_n^{\alpha,\beta})_n$ are the Hahn polynomials (see (\ref{hpol})) and
\begin{equation}\label{defsj}
s_j^u(x)=(u-x)_{j},\quad u\in \CC, j=0,1,\cdots \end{equation}
Along this paper, we use the following notation:
given a finite set of positive integers $F=\{f_1,\ldots , f_m\}$, the expression
\begin{equation}\label{defdosf}
  \begin{array}{@{}c@{}cccc@{}c@{}}
    \dosfilas{ z_{f,1} & z_{f,2} &\cdots  & z_{f,m} }{f\in F}
  \end{array}
\end{equation}
inside of a matrix or a determinant will mean the submatrix defined by
$$
\left(
\begin{array}{cccc}
z_{f_1,1} & z_{f_1,2} &\cdots  & z_{f_1,m}\\
\vdots &\vdots &\ddots &\vdots \\
z_{f_m,1} & z_{f_m,2} &\cdots  & z_{f_m,m}
\end{array}
\right) .
$$
The determinant (\ref{defmexi}) should be understood in this form. As usual $(a)_j=a(a+1)\cdots(a+j-1)$ denotes the Pochhammer symbol.

Consider now the measure $\rho _{\alpha,\beta,N}^{\F}=\eta _{\alpha,\beta,N}^{\F}(x-u_\F)$, where
\begin{equation}\label{ctmew}
\eta_{\alpha,\beta,N}^{\F}=\prod_{f\in F_1}(\lambda (x)-\lambda(f))\prod_{f\in F_2}(\lambda (x)-\lambda (f-\beta))w_{*,\alpha,\beta,N},
\end{equation}
$\lambda(x)=\lambda^{\alpha,\beta}(x)=x(x+\alpha+\beta+1)$, and $w_{*,\alpha,\beta,N}$ is the measure which respect to which the
dual Hahn polynomials $(R_n^{\alpha,\beta,N})_n$ are orthogonal  (see (\ref{dhpol}) and (\ref{masdh})).
It turns out that the sequence of polynomials $h_n^{\alpha,\beta,N;\F}$, $n\in \sigma _\F$, and the sequence of orthogonal polynomials with respect to the measure $\rho _{\alpha,\beta,N}^{\F}$ are dual sequences (see Lemma \ref{lem3.2}). As a consequence we get that the polynomials $h_n^{\alpha,\beta,N;\F}$, $n\in \sigma _\F$, are always eigenfunctions of a second order difference operator $D_\F$ (whose coefficients are rational functions); see Theorem \ref{th3.3}.

The most interesting case appears when the measure $\rho _{\alpha,\beta,N}^{\F}$ is  either positive or negative. This gives rise to the concept of Hahn admissibility (see Definition \ref{dadh} in Section \ref{sectadm}).
In Section 4, we introduce the measure
$$
\omega_{\alpha,\beta,N}^\F=\sum_{x=0}^{N-k_1} \frac{\binom{\alpha +k+x}{x}\binom{\beta +N-k_2-x}{N-k_1-x}}{\Omega_\F^{\alpha\beta,N}(x)\Omega_\F^{\alpha\beta,N}(x+1)}\delta_x,
$$
where $\Omega _\F^{\alpha, \beta, N}$ is the polynomial defined by
\begin{equation*}\label{defmexii}
\Omega _\F^{\alpha, \beta, N}(x)= \frac{ \left|
  \begin{array}{@{}c@{}lccc@{}c@{}}
    &  &&\hspace{-.9cm}{}_{1\le j\le k} \\
    \dosfilas{ h_{f}^{\alpha ,\beta ,N}(x+j-1) }{f\in F_1} \\
    \dosfilas{ s_{k-j+1}^{N-k+1}(x)s_{j-1}^{N+\beta+1}(x+j-1)h_{f}^{\alpha ,-\beta ,\beta+N}(x+j-1) }{f\in F_2}
  \end{array}
  \hspace{-.4cm}\right|}{\prod_{i=1}^{k_2}(N-x-k+1)_{k_2-i+1}(N+\beta-x-i+2)_{i-1}}.
\end{equation*}
We prove that Hahn admissibility is equivalent to the fact that this measure is either positive or negative (Lemma \ref{l3.1}).
We also  prove that if $\alpha,\beta, N$ and $\F$ are Hahn admissible, then the polynomials $h_n^{\alpha,\beta,N;\F}$, $n\in \sigma _\F$, are orthogonal and complete with respect to this measure $\omega_{\alpha,\beta,N}^\F$ (Theorem \ref{th4.5}).

In Section 5 and 6, we construct exceptional Jacobi polynomials by taking limit (in a suitable way) in the exceptional Hahn polynomials when $N\to +\infty $. We then get (see Theorem \ref{th5.1}) that for each pair $\F=(F_1,F_2)$ of finite sets of positive integers, the polynomials
\begin{equation}\label{deflaxi}
P_n^{\alpha,\beta ;\F}(x)= \frac{ \left|
  \begin{array}{@{}c@{}lccc@{}c@{}}
    & (-1)^{j-1}(P_{n-u_\F}^{\alpha ,\beta })^{(j-1)}(x) &&\hspace{-.9cm}{}_{1\le j\le k+1} \\
    \dosfilas{ (-1)^{j-1}(P_{f}^{\alpha ,\beta })^{(j-1)}(x) }{f\in F_1} \\
    \dosfilas{ (\beta-f)_{j-1}(1+x)^{k-j+1}P_{f}^{\alpha +j-1 ,-\beta-j+1 }(x)}{f\in F_2}
  \end{array}
  \hspace{-.4cm}\right|}{(1+x)^{k_2(k_2-1)}}
\end{equation}
$n\in \sigma _\F$, are eigenfunctions of a second order differential operator.

Consider the polynomial $\Omega _F^{\alpha,\beta} $ defined by
\begin{equation}\label{deflaxii}
\Omega _\F^{\alpha,\beta}(x)= \frac{ \left|
  \begin{array}{@{}c@{}lccc@{}c@{}}
    &  &&\hspace{-.9cm}{}_{1\le j\le k} \\
    \dosfilas{ (-1)^{j-1}(P_{f}^{\alpha ,\beta })^{(j-1)}(x) }{f\in F_1} \\
    \dosfilas{ (\beta-f)_{j-1}(1+x)^{k-j}P_{f}^{\alpha +j-1 ,-\beta-j+1 }(x)}{f\in F_2}
  \end{array}
  \hspace{-.4cm}\right|}{(1+x)^{k_2(k_2-1)}}.
\end{equation}
Assuming that
\begin{align}\label{3cdm}
\hbox{$\alpha +k>-1, \beta+k_1-k_2>-1$ and $\Omega _\F^{\alpha,\beta}(x)\not=0$, $x\in [-1,1]$,}
\end{align}
we prove that the polynomials $P_n^{\alpha,\beta ;\F}$, $n\in \sigma _\F$, are orthogonal with respect to the positive weight
$$
\omega_{\alpha,\beta;\F} =\frac{(1-x)^{\alpha +k}(1+x)^{\beta+k_1-k_2}}{(\Omega_\F^{\alpha,\beta}(x))^2},\quad -1<x<1,
$$
and form a complete orthogonal system in $L^2(\omega _{\alpha,\beta;\F})$ (see Theorem \ref{th6.3}).
Passing to the limit from Hahn admissibility we get the concept of Jacobi admissibility (see Definition \ref{dadj}). We prove that the assumptions (\ref{3cdm}) imply that $\alpha, \beta$ and $\F$ are Jacobi admissible (Theorem \ref{alv}). We guess that the converse is also true, but we have only been able to prove it under an additional technical condition (Theorem \ref{convth}).

We finish pointing out that, as explained above, the approach of this paper is the same as in \cite{duch} and \cite{dume} for Charlier and Hermite and Meixner and Laguerre polynomials, respectively. Since we work here with more parameters, the computations are technically more involve. Anyway, we will omit those proofs which are too similar to the corresponding ones in \cite{duch} or \cite{dume}.

\section{Preliminaries}

Let $\mu $ be a Borel measure (positive or not) on the real line. The $n$-th moment of $\mu $ is defined by
$\int _\RR t^nd\mu (t)$. When $\mu$ has finite moments for any $n\in \NN$, we can associate it a bilinear form defined in the linear space of polynomials by
\begin{equation}\label{bf}
\langle p, q\rangle =\int pqd\mu.
\end{equation}
Given an infinite set $X$ of nonnegative integers, we say that the polynomials $p_n$, $n\in X$, are orthogonal with respect to $\mu$ if they
are orthogonal with respect to the bilinear form defined by $\mu$; that is, if they satisfy
$$
\int p_np_md\mu =0, \quad n\not = m, \quad n,m \in X.
$$
When $X=\NN$ and the degree of $p_n$ is $n$, $n\ge 0$, we get the usual definition of orthogonal polynomials with respect to a measure.
When $X=\NN$, orthogonal polynomials (with non null norm) with respect to a measure are unique up to multiplication by non null constant. Let us remark  that this property is not true when $X\not =\NN$.
Positive measures $\mu $ with finite moments of any order and infinitely many points in its support has always a sequence of orthogonal polynomials $(p_n)_{n\in\NN }$, $p_n$ of degree $n$; in this case
the orthogonal polynomials have positive norm: $\langle p_n,p_n\rangle>0$. Moreover, given a sequence of orthogonal polynomials $(p_n)_{n\in \NN}$ with respect to a measure $\mu$ (positive or not) the bilinear form (\ref{bf}) can be represented by a positive measure if and only if $\langle p_n,p_n \rangle > 0$, $n\ge 0$. In Section 4 of this paper, we deal with discrete measures supported in a finite number of mass points. The following lemma will be useful in relation with this kind of measures.

\begin{lemma}\label{ldmp} Consider a discrete measure $\mu =\sum_{i=0}^N\mu_i\delta_{x_i}$, with $\mu_i\not =0$, $i=0,\cdots , N$.
\begin{enumerate}
\item If we assume that there exists a sequence $p_i$, $i=0,\cdots, N$, of orthogonal polynomials, with $\deg (p_i)=i$ and such that $\langle p_i,p_i\rangle\not =0$ has constant sign, then either $\mu_i>0$ or $\mu_i<0$, $i=0,\cdots, N$.
\item If we assume that there exists a sequence $(f_i)_{i=0}^{N+1}$ of orthogonal functions with non-null $L^2$ norm, then these functions form a basis of $L^2(\mu)$.
\end{enumerate}
\end{lemma}

\begin{proof}
(1) Assume that $\langle p_i,p_i\rangle >0$, $i=0,\cdots , N$.
For a polynomial $q\not =0$ with $\deg(q)\le N$, we have $q=\sum_{j=0}^Na_jp_j$, and at least one $a_j$ is non zero. This gives
$\langle q,q\rangle =\sum_{j=0}^Na_j^2\langle p_j,p_j\rangle >0$. Consider now the polynomial of degree $N$, $q_i(x)=\prod_{j=0,j\not =i}^N(x-x_j)$. Then $0<\langle q_i,q_i\rangle =\mu_i \prod_{j=0,j\not =i}^N(x_i-x_j)^2$. So, $\mu_i>0$.

(2) Write $A$ and $D$ for the $(N+1)\times (N+1)$ matrices defined by $A=(f_i(x_j))_{0\le i\le N;0\le j\le N}$ and $D=\diagonal(\mu_0,\cdots, \mu_N)$.
Since the functions $(f_i)_{i=0}^{N+1}$, are orthogonal, we have
$$
ADA^t=\diagonal (\langle f_i,f_i\rangle, i=0,\cdots , N).
$$
In particular $\det(ADA^t)=\prod_{i=0}^{N+1}\langle f_i,f_i\rangle\not =0$. That is, $\det (A)\not =0$. Now it is easy to conclude.
\end{proof}

When $X=\NN$, Favard's Theorem establishes that a sequence $(p_n)_{n\in \NN}$ of polynomials, $p_n$ of degree $n$, is orthogonal (with non null norm) with respect to a measure if and only if it satisfies
a three term recurrence relation of the form
$$
xp_n(x)=a_np_{n+1}(x)+b_np_n(x)+c_np_{n-1}(x), \quad n\ge 0,\quad p_{-1}=0,
$$
where $a_n, b_n$ and $c_n$, $n\in \NN$, are real numbers with $a_{n-1}c_n\not =0$, $n\ge 1$. If, in addition, $a_{n-1}c_n>0$, $n\ge 1$,
then the polynomials $(p_n)_{n\in \NN}$ are orthogonal with respect to a positive measure with infinitely many points in its support, and conversely.
Again, Favard's Theorem is not true for a sequence of orthogonal polynomials $(p_n)_{n\in X}$ when $X\not =\NN$.

To compute the degree of the exceptional polynomials introduced in this paper, we will need the following lemma.

\begin{lemma}\label{lgp1} For a pair $U,V$ of finite sets of (different) positive integers with $k_1$ and $k_2$ the number of elements of $U$ and $V$, respectively, let $R_1, R_2, \ldots, R_k,$ be nonzero polynomials satisfying that $U=\{\deg R_i,i=1,\cdots,  k_1\}$ and $V=\{\deg R_{k_1+i},i=1,\cdots,  k_2\}$. Write $r_i$ for the leading coefficient of $R_i$, $1\le i\le k=k_1+k_2$.
For real numbers $N, \beta$, consider the rational function $P$ defined by
\begin{equation}\label{defdet}
P(x)=\frac{\left|
  \begin{array}{@{}c@{}lccc@{}c@{}}
    &&&\hspace{-.9cm}{}_{1\le j\le k} \\
    \dosfilas{ R_{u}(x+j-1) }{u\in U} \\
    \dosfilas{ s_{k-j+1}^{N-k+1}(x)s_{j-1}^{N+\beta+1}(x+j-1)R_{v}(x+j-1) }{v\in V}
  \end{array}
  \hspace{-.4cm}\right|}{\prod_{i=1}^{k_2}s^{N-k+1}_{k_2-i+1}(x)s^{N+\beta-i+2}_{i-1}(x)}.
\end{equation}
Then, if $\beta+u-v\not =0$, $u\in U,v\in V$, $P$ is a polynomial of degree $\sum_{u\in U,v\in V}f-\binom{k_1}{2}-\binom{k_2}{2}$, with leading coefficient given by
\begin{equation}\label{mspcl}
p=V_{U}V_{V}\prod_{i=1}^kr_i\prod_{u\in U,v\in V}(\beta +u-v),
\end{equation}
where by $V_F$ we denote the Vandermonde determinant of $F=\{f_1,\cdots, f_k\}$
\begin{equation}\label{defvdm}
V_F=\prod_{1=i<j=k}(f_j-f_i).
\end{equation}
\bigskip
\end{lemma}

\begin{proof}
The Lemma can be proved  as  Lemma 3.3 in \cite{dudh}.
\end{proof}

We will also need the following straightforward lemma.

\bigskip
\begin{lemma}\label{rmc}
Let $M$ be a $(s+1)\times m$ matrix with $m\ge s+1$. Write $c_i$, $i=1,\ldots , m$, for the columns of $M$ (from left to right). Assume that for $0\le j\le m-s-1$ the consecutive columns $c_{j+i}$, $i=1,\cdots ,s$, of $M$ are linearly independent while the consecutive columns $c_{j+i}$, $i=1,\cdots ,s+1$, are linearly dependent. Then $\rank M=s$.
\end{lemma}

\subsection{Christoffel transform}\label{secChr}
Let $\mu$ be a measure (positive or not) and assume that $\mu$ has a sequence of orthogonal polynomials
$(p_n)_{n\in \NN}$, $p_n$ with degree $n$ and $\langle p_n,p_n\rangle \not =0$ (as we mentioned above, that always happens if $\mu$ is positive, with finite moments and infinitely many points in its support).

Given a finite set $F$ of real numbers, $F=\{f_1,\cdots , f_k\}$, $f_i<f_{i+1}$, we write $\Phi_n$, $n\ge 0$, for the $k\times k$ determinant
\begin{equation}\label{defph}
\Phi_n=\vert p_{n+j-1}(f_i)\vert _{i,j=1,\cdots , k}.
\end{equation}
Notice that $\Phi_n$, $n\ge 0$, depends on both, the finite set $F$ and the measure $\mu$. In order to stress this dependence, we sometimes write  $\Phi_n^{\mu, F}$ for $\Phi_n$.

Along this Section we assume that the set $\X_\mu^F=\{ n\in \NN :\Phi_n^{\mu,F}=0\}$ is finite. We denote $\x_\mu ^F=\max \X_\mu ^F$. If
$\X_\mu ^F=\emptyset$ we take $\x_\mu ^F=-1$.

The Christoffel transform of $\mu$ associated to the annihilator polynomial $\pp$ of $F$,
$
\pp (x)=(x-f_1)\cdots (x-f_k),
$
is the measure defined by $ \mu_F =\pp \mu$.

Orthogonal polynomials with respect to $\mu_F$ can be constructed by means of the formula
\begin{equation}\label{mata00}
q_n(x)=\frac{1}{\pp (x)}\det \begin{pmatrix}p_n(x)&p_{n+1}(x)&\cdots &p_{n+k}(x)\\
p_n(f_1)&p_{n+1}(f_1)&\cdots &p_{n+k}(f_1)\\
\vdots&\vdots&\ddots &\vdots\\
p_n(f_k)&p_{n+1}(f_k)&\cdots &p_{n+k}(f_k) \end{pmatrix}.
\end{equation}
Notice that the degree of $q_n$ is equal to $n$ if and only if $n\not\in \X_\mu ^F$. In that case the leading coefficient $\lambda^Q_n$ of $q_n$ is
equal to $(-1)^k\lambda^P_{n+k}\Phi_n$, where $\lambda ^P_n$ denotes the leading coefficient of $p_n$.
The next Lemma follows easily using \cite{Sz}, Th. 2.5.

\begin{lemma}\label{sze} Let $K$ be a positive integer or infinity.
The measure $\mu_F$ has a sequence $(q_n)_{n=0}^K $, $q_n$ of degree $n$, of orthogonal polynomials if and only if $\X_\mu ^F=\emptyset$ or $\min \X_\mu>K$.
In that case, an orthogonal polynomial of degree $n$ with respect to $\mu _F$ is given by (\ref{mata00}) and also $\langle q_n,q_n\rangle _{\mu _F}\not =0$, $0\le n\le K$. If $\X_\mu \not =\emptyset$, the polynomial $q_n$ (\ref{mata00}) has still degree $n$ for $n\not \in \X_\mu^F$, and satisfies $\langle q_n,r\rangle_{\mu _F}=0$ for all polynomial $r$ with degree less than $n$ and $\langle q_n,q_n\rangle _{\mu _F}\not =0$.
\end{lemma}

From (\ref{mata00}), one can also deduce (see Lemma 2.8 of \cite{duch})
\begin{equation}\label{n2q}
\langle q_n,q_n\rangle _{\mu_F}=(-1)^k\frac{\lambda^P_{n+k}}{\lambda^P_{n}}\Phi_n\Phi_{n+1}\langle p_n,p_n\rangle _{\mu},\quad n> \x_\mu^F+1.
\end{equation}
This identity holds for $n\ge 0$ when $\X_\mu =\emptyset$

\subsection{Finite sets and pair of finite sets of positive integers.}\label{sfspi}
Consider the set $\Xi$  formed by all finite sets of positive
integers:
\begin{align*}
\Xi=\{F:\mbox{$F$ is a finite set of positive integers}\} .
\end{align*}
We consider the involution $I$ in $\Xi$ defined by
\begin{align}\label{dinv}
I(F)=\{1,2,\cdots, \max F\}\setminus \{\max F-f,f\in F\}.
\end{align}
The definition of $I$ implies that $I^2=Id$.

The set $I(F)$ will be denoted by $G$: $G=I(F)$. Notice that
$$
\max F=\max G,\quad m=\max F-k+1,
$$
where $k$ and $m$ are the number of elements of $F$ and $G$,
respectively (for more details see \cite{duch}, Section 2.3).

For a finite set $F=\{f_1,\cdots ,f_k\}$, $f_i<f_{i+1}$, of
positive integers, we define
\begin{align}\label{defs0}
s_F&=\begin{cases} 1,& \mbox{if $F=\emptyset$},\\
k+1,&\mbox{if $F=\{1,2,\cdots , k\}$},\\
\min \{s\ge 1:s<f_s\}, & \mbox{if $F\not =\{1,2,\cdots k\}$},
\end{cases}\\\label{defffd}
F_{\Downarrow}&=\begin{cases} \emptyset,& \mbox{if $F=\emptyset$ or $F=\{1,2,\cdots , k\}$,}\\
\{f_{s_F}-s_F,\cdots , f_k-s_F\},& \mbox{if $F\not =\{1,2,\cdots ,
k\}$}.
\end{cases}
\end{align}

From now on, $\F=(F_1,F_2)$ will denote a pair of finite sets of
positive integers. We will write $F_1=\{ f_1^{ 1\rceil },\cdots ,
f_{k_1}^{1\rceil}\}$, $F_2=\{ f_1^{2\rceil},\cdots , f_{k_2}^{2\rceil}\}$, with
$f_i^{j\rceil}<f_{i+1}^{j\rceil}$ (the use of $f_i^2$ to describe elements of $F_2$ is confusing because it looks like a square, this is the reason
why we use the notation $f_i^{2\rceil}$).
 Hence $k_j$ is the number of elements of $F_j$,
$j=1,2$, and $k=k_1+k_2$ is the number of elements of $\F$. One of
the components of $\F$, but not both, can be the empty set.

We associate to $\F$ the nonnegative integers $u_\F$ and $w_\F$ and the infinite set of nonnegative integers $\sigma_\F$ defined by
\begin{align}\label{defuf}
u_\F&=\sum_{f\in F_1}f+\sum_{f\in
F_2}f-\binom{k_1+1}{2}-\binom{k_2}{2},\\\label{defwf2}
w_\F&=\sum_{f\in F_1}f+\sum_{f\in
F_2}f-\binom{k_1}{2}-\binom{k_2}{2}+1,\\\label{defsf}
\sigma _\F&=\{u_\F,u_\F+1,u_\F+2,\cdots \}\setminus \{u_\F+f,f\in
F_1\}.
\end{align}
The infinite set $\sigma_\F$ will be the set of indices for the exceptional Hahn or Jacobi polynomials associated to $\F$.
Notice that $w_\F=u_\F+k_1+1$.

For a pair $\F=(F_1,F_2)$ of positive integers we denote by
$\F_{j,\{ i\}}$, $i=1,\ldots , k_j$, $j=1,2$, and
$\F_{\Downarrow}$ the pair of finite sets of positive integers
defined by
\begin{align}\label{deff1}
\F_{1,\{ i\} }&=(F_1\setminus \{f_i^{ 1\rceil }\},F_2),\\\label{deff2}
\F_{2,\{ i\} }&=(F_1,F_2\setminus \{f_i^{ 2\rceil }\}),\\\label{deffd}
\F_{\Downarrow}&=((F_1)_{\Downarrow},F_2),
\end{align}
where $(F_1)_{\Downarrow}$ is defined by (\ref{defffd}). We also define
\begin{equation}\label{defs0f}
s_\F=s_{F_1}
\end{equation}
where the number $s_{F_1}$ is defined by (\ref{defs0}).

\subsection{Admissibility}\label{sectadm}
Using the determinant (\ref{defmexi}), whose entries are Hahn polynomials, $h_n^{\alpha,\beta, N}$, we will associate to each pair $\F$ of finite sets of positive integers a sequence of polynomials which are  eigenfunctions of a second order difference operator.
The more important of these examples are those which, in addition, are orthogonal and complete with respect to a positive measure. As it happens with the Hahn family, for the existence of such a positive measure the parameter $N$ must be taken to be a positive integer: we make that assumption along this section.

The key concept for the construction of exceptional Hahn and Jacobi polynomials is that of admissibility.
The analogy with the cases of exceptional Charlier and Hermite, and Meixner and Laguerre polynomials suggests that
the admissibility condition in the Hahn and Jacobi case should be equivalent to either the positivity or the negativity of the measure $\rho_{\alpha,\beta,N}^\F$ (\ref{ctmew}).

To avoid division by zero or trivial situations, along this section we will assume
\begin{equation}\label{cpara}
\alpha,\beta, \alpha +\beta\not=-1,-2,\cdots ,-N,\quad \{0,1,\cdots ,N\}\setminus (F_1\cup (-\beta+F_2))\not =\emptyset.
\end{equation}

\begin{definition}\label{dadh} Let $\F=(F_1,F_2)$ be  a pair of finite sets of positive integers. For a positive integer $N$,
real numbers $\alpha, \beta$ satisfying (\ref{cpara}) and $x\in \NN$, write
\begin{align}\label{defadmh}
\HH ^{\alpha,\beta,N}_\F(x)=\prod_{f\in F_1}(x-f)&(x+f+\alpha+\beta+1)\prod_{f\in F_2}(x+\beta-f)(x+f+\alpha+1)\\\nonumber &\times\frac{(2x+\alpha+\beta+1)(\alpha+1)_x}{(x+\alpha+\beta+1)_{N+1}(\beta+1)_x}.
\end{align}
We say that $\alpha, \beta, N$ and $\F$ are Hahn admissible if $\HH^{\alpha,\beta,N}_\F(x)$ has constant sign for $x=0,\cdots , N $.
\end{definition}

For $u\in \RR$, we will use the notation $\hat u=\max \{-[u],0\}$ (where $[u]$ denotes the value of the floor function at $u$, i.e. $[u]=\max\{s\in \ZZ: s\le u\}$). Notice that always $u+\hat u\ge 0$.
If we write
$$
(x+\alpha+\beta+1)_{N+1}=(x+\alpha+\beta+1)_{\widehat{\alpha+\beta+1}}(x+\alpha+\beta+1+\widehat{\alpha+\beta+1})_{N+1-\widehat{\alpha+\beta+1}},
$$
we see that if $N>\widehat{\alpha+\beta+1}$ then
\begin{equation}\label{ndN}
\sign( (x+\alpha+\beta+1)_{N+1})=\sign((x+\alpha+\beta+1)_{\widehat{\alpha+\beta+1}}).
\end{equation}
This means that for $N>\widehat{\alpha+\beta+1}$ the sign of $\HH ^{\alpha,\beta,N}_\F(x)$ does not depend on $N$. Hence we have proved the following Lemma.

\begin{lemma}\label{bbl1} If for some $N_0>\widehat{\alpha+\beta+1}$, $\alpha,\beta,N_0$ and $\F$ are Hahn admissible, then for any $N>\widehat{\alpha+\beta+1}$, $\alpha,\beta,N$ and $\F$ are Hahn admissible as well.
\end{lemma}

We define the Jacobi admissibility as follows.

\begin{definition}\label{dadj} Let $\F=(F_1,F_2)$ be a pair of finite sets of positive integers. For
real numbers $\alpha, \beta$ satisfying
\begin{equation}\label{cpara2}
\alpha,\beta, \alpha +\beta\not=-1,-2,\cdots ,
\end{equation}
and $x\in \NN$, write
\begin{align}\label{defadmj}
\HH^{\alpha,\beta}_\F(x)=\prod_{f\in F_1}(x-f)&(x+f+\alpha+\beta+1)\prod_{f\in F_2}(x+\beta-f)(x+f+\alpha+1)\\\nonumber &\times\frac{\Gamma(x+\alpha+1)\Gamma(x+\beta+1)}{(2x+\alpha+\beta+1)\Gamma(x+\alpha+\beta+1)}.
\end{align}
We say that $\alpha, \beta $ and $\F$ are Jacobi admissible if $\HH^{\alpha,\beta}_\F(x)\ge 0$ for $x\in \NN $.
\end{definition}

Notice that the condition $x\in \NN$ can be changed to
$$
x\in \{0,1,\cdots , \max(\max F_1,-[\beta]+\max F_2,-[\alpha],-[\beta],-[\alpha+\beta])+1\}.
$$
\begin{remark}\label{rm1}
Since for $u\in \RR$
\begin{align*}
(x+u)_{N+1}&=(x+u)_{\widehat{u}}(x+u+\widehat{u})_{N+1-\widehat{u}}=\frac{\Gamma(x+u+\hat u)(x+u+\widehat{u})_{N+1-\widehat{u}}}{\Gamma(x+u)},\\
(u)_x&=\frac{\Gamma(x+u)}{\Gamma(u)},
\end{align*}
we get $\sign (\HH^{\alpha,\beta,N}_\F(x))=\sign (\HH^{\alpha,\beta}_\F(x))\sign (\Gamma(\alpha+1)\Gamma(\beta+1))$. Hence
for $N>\widehat{\alpha+\beta+1}$  the Jacobi admissibility of $\alpha,\beta$ and $\F$ implies the Hahn admissibility of
$\alpha, \beta , N$ and $\F$. Part 1 and 2 of Lemma \ref{ladm} below show  that the Hahn admissibility of $\alpha, \beta , N$ and $\F$ for some $N>\widehat{\alpha+\beta+1}$ also implies the Jacobi admissibility of $\alpha,\beta$ and $\F$.
\end{remark}

This admissibility concept is more involve than the corresponding for exceptional Charlier and Hermite or Meixner and Laguerre polynomials
(see \cite{duch}, p. 31 and \cite{dume}, definition 1.2, respectively). The admissibility depends now on two (Jacobi) or three (Hahn) parameters, while for Charlier and Hermite no parameter is involved and for Meixner and Laguerre, only one parameter is involved.

We have not found in the literature a definition as (\ref{dadh}) or (\ref{dadj}) for Hahn and Jacobi admissibility, respectively.

In the following Lemma we include some  consequences derived from the definitions \ref{dadh} and \ref{dadj}.

\begin{lemma}\label{ladm} Given a positive integer $N$, real numbers $\alpha,\beta$ satisfying (\ref{cpara}) and a pair $\F $ of finite sets of positive integers, we have
\begin{enumerate}
\item assume that either $\alpha, \beta$ and $\F$ are Jacobi admissible or $\alpha, \beta, N$ and $\F$ are Hahn admissible for $N>\widehat{\alpha+\beta+1}$,  then $\alpha+k>-1$ and $\beta+k_1-k_2>-1$.
\item If $\alpha,\beta, N$ and $\F$ are Hahn admissible and $N>\widehat{\alpha+\beta+1}$ then the sign of $(\alpha +1)_k(\beta+1)_{k_1-k_2}$ is equal to the sign of $\HH ^{\alpha,\beta,N}_\F$.
\item If $\alpha,\beta$ and $\F$ are Jacobi admissible and $F_1=\emptyset$ then $\alpha+1,\beta+1$ and $\F$ are Jacobi admissible as well.
\item Assume that $\alpha,\beta$ and $\F$ are Jacobi admissible and that $s_\F+\alpha+\beta+1>0$. Then $\alpha+s_\F,\beta+s_\F$ and $\F_\Downarrow$ are Jacobi admissible as well.
\end{enumerate}
\end{lemma}

\begin{proof}

\noindent
Both parts 1 and 2 of this Lemma will be proved just after the proof of Part 1 of Theorem \ref{convth}.

\noindent
Proof of part 3.  Indeed, a direct computation gives
\begin{equation}\label{adm6}
\HH _{\F}^{\alpha+1,\beta+1}(x)=\frac{\HH _{\F}^{\alpha,\beta}(x+1)}{x+\alpha+\beta+2}.
\end{equation}
Since $F_1=\emptyset$ (and hence $k_1=0$), using part 1 of this Lemma, we have $\alpha+k_2> -1$ and $\beta-k_2>-1$, from where we get
$\alpha+\beta>-2$. This shows that for $x\in \NN$, $x+\alpha+\beta+2>0$. Since $\alpha,\beta$ and $\F$ are Jacobi admissible, (\ref{adm6}) shows that
$\alpha+1, \beta+1$ and $\F$ are also Jacobi admissible.

\noindent
Proof of part 4. For a finite set of positive integers $F$, (\ref{defffd}) gives that
$$
F=\begin{cases} s_F+F_\Downarrow,& \mbox{for $s_F=1$}\\
\{1,\cdots, s_F-1\}\cup (s_F+F_\Downarrow),& \mbox{for $s_F>1$}.
\end{cases}
$$
This gives
$$
\prod_{f\in F_\Downarrow}(x-f)(x+f+2s_F+\alpha+\beta+1)=\frac{\prod_{f\in F}(x+s_F-f)(x+f+s_F+\alpha+\beta+1)}{\prod_{j=1}^{s_F-1}(x+s_F-j)(x+s_F+j+\alpha+\beta+1)}.
$$
Hence, after straightforward computation, we can write
\begin{align*}\label{adm5}
\HH _{\F_\Downarrow}^{\alpha+s_\F,\beta+s_\F}(x)&=\frac{\HH _{\F}^{\alpha,\beta}(x+s_\F)}{\prod_{j=1}^{s_\F-1}(x+s_\F-j)(x+s_\F+\alpha+\beta+1+j)^2}\\&\quad\quad\times \frac{1}{(x+s_\F+\alpha+\beta +1)}.
\end{align*}
This shows that if $\alpha,\beta$ and $\F$ are Jacobi admissible and $s_\F+\alpha+\beta +1>0$, then
$\alpha+s_\F, \beta+s_\F$ and $\F_\Downarrow$ are also Jacobi admissible.

\end{proof}

\subsection{Dual Hahn, Hahn and Jacobi polynomials}
We include here basic definitions and facts about dual Hahn, Hahn and Jacobi polynomials, which we will need in the following Sections.

For $\alpha $ and $\beta$ real numbers, we write
\begin{equation}\label{deflamb}
\lambda^{\alpha,\beta}(x)=x(x+\alpha+\beta+1).
\end{equation}
To simplify the notation we sometimes write $\lambda(x)=\lambda^{\alpha,\beta}(x)$.

For $\alpha \not =-1,-2,\cdots $ we write $(R_{n}^{\alpha,\beta,N})_n$ for the sequence of dual Hahn polynomials defined by
\begin{equation}\label{dhpol}
R_{n}^{\alpha,\beta,N}(x)=\sum _{j=0}^n(-1)^j\frac{(-n)_j(-N+j)_{n-j}}{(\alpha+1)_jj!}\prod_{i=0}^{j-1}(x-i(\alpha+\beta+1+i))
\end{equation}
(we have taken a slightly different normalization from the one used in \cite{KLS}, pp, 234-7 from where
the next formulas can be easily derived). Notice that $R_{n}^{\alpha,\beta,N}$ is always a polynomial of degree $n$.
Using that $(-1)^j\prod_{i=0}^{j-1}(\lambda^{\alpha,\beta}(x)-i(\alpha+\beta+1+i))=(-x)_j(x+\alpha+\beta+1)_j$, we get the hypergeometric representation
$$
R_{n}^{\alpha,\beta,N}(\lambda^{\alpha,\beta}(x))=(-N)_n\pFq{3}{2}{-n,-x,x+\alpha+\beta+1}{\alpha+1,-N}{1}.
$$
When $N$ is a positive integer then the polynomial $R_{n}^{\alpha,\beta,N}(x)$ for $n\ge N+1$ is always divisible by $\prod_{i=0}^{N}(x-i(\alpha+\beta+1+i))$. Hence
\begin{equation}\label{mdc1}
R_{n}^{\alpha,\beta,N}(\lambda^{\alpha,\beta}(i))=0,\quad n\ge N+1, i=0,\cdots, N.
\end{equation}
Dual Hahn polynomials  satisfy the following three term recurrence formula
\begin{align}\label{Mxttrr}
xR_n&=A_nR_{n+1}+B_nR_n+C_nR_{n-1}, \quad n\ge 0,\quad R_{-1}=0,\\\nonumber
A_n&=n+\alpha+1,\\\nonumber
B_n&=-(n+\alpha+1)(n-N)-n(n-\beta-N-1),\\\nonumber
C_n&=n(n-\beta-N-1)(n-N-1).
\end{align}
(to simplify the notation we remove the parameters in some formulas).
Hence, when $N$ is not a nonnegative integer and $\alpha ,-\beta-N-1 \not =-1,-2,\cdots $, they are always orthogonal with respect to a moment functional $w_{*,\alpha,\beta,N}$. When $N$ is a positive integer and $\alpha ,\beta \not =-1,-2,\cdots -N $, $\alpha+\beta \not=-1,\cdots, -2N-1$, we have
\begin{align}\label{masdh}
w_{*;\alpha,\beta,N}&=\sum _{x=0}^N \frac{(2x+\alpha+\beta+1)(\alpha+1)_x(-N)_xN!}{(-1)^x(x+\alpha+\beta+1)_{N+1}(\beta+1)_xx!}\delta_{\lambda(x)},
\\\label{normedh}
\langle R_n^{\alpha,\beta,N},R_n^{\alpha,\beta,N}\rangle &=\frac{(-N)_n^2}{\binom{\alpha+n}{n}\binom{\beta+N-n}{N-n}},\quad  n\in \NN .
\end{align}
Notice that $\langle R_n^{\alpha,\beta,N},R_n^{\alpha,\beta,N}\rangle\not =0$ only for $0\le n\le N$.
The moment functional $w_{*,\alpha,\beta,N}$ can be represented by either a positive or a negative measure only when $N$ is a positive integer
and either $-1<\alpha,\beta$ or $\alpha,\beta<-N$, respectively.

Dual Hahn polynomials satisfy the following identity
\begin{equation}\label{sdm3}
R_{n}^{\alpha,\beta,N}(\lambda^{\alpha,\beta}(x-\beta))=R_{n}^{\alpha,-\beta,\beta+N}(\lambda^{\alpha,-\beta}(x)).
\end{equation}
\bigskip

For $\alpha, \alpha+\beta \not =-1,-2,\cdots $ we write $(h_{n}^{\alpha,\beta,N})_n$ for the sequence of Hahn polynomials defined by
\begin{equation}\label{hpol}
h_{n}^{\alpha,\beta,N}(x)=(-N)_n\pFq{3}{2}{-n,-x,n+\alpha+\beta+1}{\alpha+1,-N}{1}
\end{equation}
(we have taken a slightly different normalization from the one used in \cite{KLS}, pp, 234-7 from where
the next formulas can be easily derived). Notice that $h_{n}^{\alpha,\beta,N}$ is always a polynomial of degree $n$. When $N$ is a positive integer then
\begin{equation}\label{mdc2}
\mbox{the polynomial $h_{n}^{\alpha,\beta,N}$ for $n\ge N+1$ is always divisible by $(-x)_{N+1}$}.
\end{equation}
The hypergeometric representation of dual Hahn and Hahn polynomials show the duality
\begin{equation}\label{sdm2b}
(-N)_m h_{n}^{\alpha,\beta,N}(m)=(-N)_n R_{m}^{\alpha,\beta,N}(\lambda^{\alpha,\beta}(n)),\quad n,m\ge 0.
\end{equation}
If $N$ is not a nonnegative integer and $\alpha ,\beta,\alpha+\beta,\alpha+\beta+N \not =-1,-2,\cdots $, they are always orthogonal with respect to a moment functional $\rho_{\alpha,\beta,N}$. For $N$ a positive integer and $\alpha ,\beta \not=-1,\cdots -N $, $\alpha+\beta\not=-1,\cdots, -2N-1$  we have
\begin{align}\label{hw}
\rho_{\alpha,\beta,N}&=\sum _{x=0}^N \binom{x+\alpha}{x}\binom{\beta+N-x}{N-x}\delta_{x},\\\label{normeh}
\langle h_n^{\alpha,\beta,N},h_n^{\alpha,\beta,N}\rangle &=\frac{(-N)_n^2}{w_{*;\alpha,\beta,N}(n)},\quad  n\in \NN,
\end{align}
where $w_{*;\alpha,\beta,N}(n)$ is the mass of the dual Hahn weight at $\lambda(n)$ given by (\ref{masdh}). Notice that
$\langle h_n^{\alpha,\beta,N},h_n^{\alpha,\beta,N}\rangle\not =0$ only for $0\le n\le N$.
The moment functional $\rho_{\alpha,\beta,N}$ can be represented by either a positive or a negative measure only when $N$ is a positive integer and either $-1<\alpha,\beta$ or $\alpha,\beta<-N$, respectively.

Hahn polynomials satisfy the following identities
\begin{align}\label{sdm}
h_{n}^{\alpha,\beta, N}(x+1)-h_{n}^{\alpha,\beta, N}(x)&=\frac{\lambda^{\alpha,\beta}(n)}{\alpha+1}h_{n-1}^{\alpha+1,\beta+1, N-1}(x),\\ \label{sdm2}
s^{N-k+1}_{k+1-j}(x)s^{N+\beta-j+2}_{j-1}(x)h_{n}&^{\alpha,-\beta,\beta+ N}(x+j-1)-s^{N-k+1}_{k+2-j}(x)s^{N+\beta-j+3}_{j-2}(x)h_{n}^{\alpha,-\beta, \beta+N}(x+j-2)\\ \nonumber=&\frac{(\alpha+n+1)(\beta-n)}{\alpha+1}s^{N-k+1}_{k+1-j}(x)s^{N+\beta-j+3}_{j-2}(x)h_{n}^{\alpha+1,-\beta-1,\beta+ N}(x+j-2),
\end{align}
where $s_j^u$ is the polynomial of degree $j$ defined in (\ref{defsj}).

For $\alpha,\beta \in \RR , \alpha,\beta \not=-1,-2,\cdots$, we use the standard definition of the Jacobi polynomials $(P_{n}^{\alpha,\beta})_n$
\begin{equation}\label{defjac}
P_{n}^{\alpha,\beta}(x)=2^{-n}\sum _{j=0}^n \binom{n+\alpha}{j}\binom{n+\beta}{n-j}(x-1)^{n-j}(x+1)^{j}
\end{equation}
(see \cite{EMOT}, pp. 169-173 and also \cite{KLS}, pp. 216-221).

For $\alpha,\beta, \alpha+\beta \not =-1,-2,\cdots$,
they are orthogonal with respect to a measure $\mu _{\alpha,\beta}=\mu _{\alpha,\beta}(x)dx$, which it is positive only when
$\alpha ,\beta >-1$, and then
\begin{equation}\label{jacw}
\mu_{\alpha,\beta}(x) =(1-x)^\alpha (1+x)^{\beta}, \quad -1<x<1.
\end{equation}
We will  use the following formulas
\begin{equation}\label{Lagder}
    \left(P_n^{\alpha,\beta}\right)'=\frac{n+\alpha+\beta+1}{2}P_{n-1}^{\alpha+1,\beta+1},
\end{equation}
\begin{equation}\label{Lagab}
   ((1+x)P_n^{\alpha,-\beta}(x))'=(\beta+1)
   P_n^{\alpha,-\beta}(x)-(\beta-n)P_n^{\alpha+1,-\beta-1}(x).
\end{equation}
One can obtain Jacobi polynomials from Hahn polynomials using the limit
\begin{equation}\label{blmel}
\lim_{N\to +\infty}\frac{h_n^{\alpha,\beta,N}\left(\frac{(1-x)N}{2}\right)}{(-N)_n}=\frac{n!P_n^{\alpha,\beta}(x)}{(\alpha+1)_n}
\end{equation}
see \cite{KLS}, p. 207 (note that we are using for
Hahn polynomials a different normalization to that in \cite{KLS}). This limit is uniform in compact sets of $\CC$.

\section{Constructing polynomials which are eigenfunctions of second order difference operators}
We assume a number of constrains on the parameters $\alpha, \beta ,N$ and the pair $\F=(F_1,F_2)$ of finite sets of positive integers. We always assume
\begin{equation}\label{cpar11}
\alpha,\beta,\alpha +\beta\not =-1,-2,\cdots .
\end{equation}
In addition, we also assume
\begin{equation}\label{cpar12}
\alpha-\beta,\beta-f_{k_2}^{2\rceil}-1\not =-1,-2,\cdots ,\qquad \mbox{if $F_2\not =\emptyset$}
\end{equation}
(let us recall that $f_{k_2}^{2\rceil}$ is the maximum number in $F_2$). These assumptions are needed to define the polynomials (\ref{defmex}) below, and to guarantee that the polynomial $h_n^{\alpha, \beta, N;\F}$ has degree $n$
(actually the assumption $\beta-f_{k_2}^{2\rceil}-1\not =-1,-2,\cdots $ can be changed to the weaker one $\beta \not \in (F_2-F_1)\cup \cup_{n\in \sigma_\F}(-n+u_\F+F_2)$). We do not need to assume at this stage that $N$ is a positive integer.

\begin{definition}
Let $\F =(F_1,F_2)$ be a pair of finite sets of positive integers.
We define the polynomials $h_n^{\alpha, \beta, N;\F}$, $n\in \sigma _\F$, as
\begin{equation}\label{defmex}
h_n^{\alpha, \beta, N;\F}(x)= \frac{ \left|
  \begin{array}{@{}c@{}lccc@{}c@{}}
    & h_{n-u_\F}^{\alpha ,\beta ,N}(x+j-1) &&\hspace{-.9cm}{}_{1\le j\le k+1} \\
    \dosfilas{ h_{f}^{\alpha ,\beta ,N}(x+j-1) }{f\in F_1} \\
    \dosfilas{ s_{k-j+1}^{N-k+1}(x)s_{j-1}^{N+\beta+1}(x+j-1)h_{f}^{\alpha ,-\beta ,\beta+N}(x+j-1) }{f\in F_2}
  \end{array}
  \hspace{-.4cm}\right|}{\prod_{i=1}^{k_2}(N-x-k+1)_{k_2-i}(N+\beta-x-i+2)_{i-1}}
\end{equation}
where $s_j^u$ is the polynomial of degree $j$ defined by $s_j^u(x)=(u-x)_j$ (see (\ref{defsj})) and the number $u_\F$ and the infinite set of nonnegative integers $\sigma _\F$ are defined by (\ref{defuf}) and (\ref{defsf}), respectively.
\end{definition}
The determinant (\ref{defmex}) should be understood as explained in  (\ref{defdosf}).

To simplify the notation, we will sometimes write $h_n^\F=h_n^{\alpha,\beta,N;\F}$.

Using Lemma \ref{lgp1}, we deduce that $h_n^\F$, $n\in \sigma _\F$, is a polynomial of degree $n$ with leading coefficient equal to
\begin{equation*}\label{lcrn}
V_{F_1}V_{F_2}\prod_{i\in \{n-u_\F\},F_1}r_i^{\alpha,\beta}\prod_{i\in F2}r_i^{\alpha,-\beta}\prod_{i\in \{n-u_\F\},F_1,f_2\in F_2}(\beta +i-f_2)\prod_{f\in F_1}(f-n+u_\F),
\end{equation*}
where $V_F$ is the Vandermonde determinant (\ref{defvdm}) and $r_i^{a,b}=\frac{(a+b+i+1)_{i}}{(a+1)_{i}}$, that is, the leading coefficient of the Hahn polynomial $h_i^{a,b,N}$. The assumptions (\ref{cpar11}) and (\ref{cpar12})  imply that the above leading coefficient does not vanish when $n\in \sigma_\F$.

With the convention that $h_n=0$ for $n<0$, the determinant (\ref{defmex}) defines a polynomial for any $n\ge 0$, but for $n\not \in \sigma_\F$ we have $h_n^\F=0$.

The most interesting case appears when $N$ is a positive integer. In this case, as a consequence of (\ref{mdc2}), for $n\ge N+u_\F+1$, the polynomial $h_n^\F$ is always divisible by $(-x)_{N-k_1+1}$.

Combining columns in (\ref{defmex}) and taking into account (\ref{sdm}) and (\ref{sdm2}), we have the alternative definition
\begin{equation}\label{defmexa}
h_n^{\alpha, \beta, N;\F}(x)= \frac{ \left|
  \begin{array}{@{}c@{}lccc@{}c@{}}
    & \prod_{i=0}^{j-2}(\lambda (n-u_\F)-\lambda(i))h_{n-u_\F-j+1}^{\alpha +j-1,\beta +j-1,N-j+1}(x)&&\hspace{-1.2cm}{}_{1\le j\le k+1} \\
    \dosfilas{\prod_{i=0}^{j-2}(\lambda (f)-\lambda(i)) h_{f-j+1}^{\alpha +j-1,\beta +j-1,N-j+1}(x) }{f\in F_1} \\
    \dosfilas{ (\alpha +f+1,\beta -f)_{j-1}s_{k-j+1}^{N-k+1}(x)h_{f}^{\alpha +j-1,-\beta -j+1 ,\beta+N}(x) }{f\in F_2}
  \end{array}
  \right|}{\prod_{i=0}^{k}(\alpha+1)_i\prod_{i=1}^{k_2}(N-x-k+1)_{k_2-i}(N+\beta-x-i+2)_{i-1}},
\end{equation}
where we are using the notation $
(a,b,\cdots,c)_j=(a)_j(b)_j\cdots (c)_j$. That expression can be rewritten using that
$
\prod_{i=0}^{j-2}(\lambda (n)-\lambda(i))=(n-j+2,n+\alpha+\beta+1)_{j-1}.
$

The  polynomials $h_n^\F$, $n\in \sigma_\F$, are strongly related by duality with the polynomials $q_n^\F$, $n\ge 0$, defined by
\begin{equation}\label{defqnme}
q_n^{\F}(x)= \frac{ \left|
  \begin{array}{@{}c@{}lccc@{}c@{}}
    & R_{n+j-1}^{\alpha ,\beta ,N}(x)&&\hspace{-.9cm}{}_{1\le j\le k+1} \\
    \dosfilas{R_{n+j-1}^{\alpha ,\beta ,N}(\lambda^{\alpha,\beta}(f))}{f\in F_1} \\
    \dosfilas{R_{n+j-1}^{\alpha ,-\beta ,\beta+N}(\lambda^{\alpha,-\beta}(f)) }{f\in F_2}
  \end{array}
  \hspace{-.3cm}\right|}{\prod_{f\in F_1}(x-\lambda^{\alpha,\beta}(f))\prod_{f\in F_2}(x-\lambda^{\alpha,\beta}(f-\beta))}.
\end{equation}

Notice that when $N$ is a positive integer, $F_1\not =\emptyset$ and $F_1\subset \{1,\cdots, N\}$, then $q_n^\F(x)=0$ for $n\ge N-k_1+2$. Indeed, if $n\ge N-k_1+2$ then for $j\ge k_1$, we have
$n+j-1\ge N+1$ and hence $R_{n+j-1}^{\alpha ,\beta ,N}(\lambda^{\alpha,\beta}(f))=0$ (see (\ref{mdc1})).

When $N$ is a positive integer, and under suitable conditions on the parameters, one can see using  (\ref{sdm3}) and Lemma \ref{sze}, that the  polynomials  $q_n^\F$, $n=0,\ldots, N-k$, are orthogonal with respect to the measure
\begin{equation}\label{mraf}
\rho _{\alpha,\beta, N}^{\F}=\sum _{x=u_\F}^{N+u_\F} \prod_{f\in F_1}(\lambda (x-u_\F)-\lambda(f))\prod_{f\in F_2}(\lambda(x-u_\F)-\lambda(f-\beta))w_{*;\alpha,\beta,N}(x-u_\F)\delta_{\lambda(x-u_\F)},
\end{equation}
where $w_{*;\alpha,\beta,N}(x)$ is the mass at $\lambda (x)$ of the dual Hahn weight given by (\ref{masdh}).

\begin{lemma}\label{lem3.2}
If $u\ge 0$ and $v\in \sigma_\F$, then
\begin{equation}\label{duaqnrn}
\kappa \zeta_vq_u^\F(\lambda^{\alpha,\beta}(v-u_\F))=\xi_uh_v^\F(u),
\end{equation}
where
\begin{align*}
\kappa&=\prod_{f\in F_1}(-N)_f\prod_{f\in F_2}(-N-\beta)_f,\\
\xi_u&=(-1)^{(k+1)(u+k/2)}(N-u+1)_u^{k_1+1}\prod_{i=0}^{k_1-1}(N-u-k_1+1+i)_{k_1-i}\\&\hspace{2.5cm}
\times (N+\beta-u+1)_u^{k_2}\prod_{i=0}^{k_2-1}(N+\beta-u-i+1)_{i},\\
\zeta_v&=(-N)_{v-u_\F}\prod_{f\in F_1}(\lambda(v-u_\F)-\lambda(f))\prod_{f\in F_2}(\lambda(v-u_\F)-\lambda(f-\beta)).
\end{align*}
\end{lemma}

\begin{proof}
It is a straightforward consequence of the duality (\ref{sdm2b}) for the Hahn and dual Hahn polynomials.

\end{proof}

We now prove that the polynomials $h_n^\F$, $n\in \sigma_\F$, are eigenfunctions of a second order difference operator. To establish the result in full, we need some more notations.
We denote by $\Omega _\F ^{\alpha,\beta,N}(x)$ and $\Lambda _\F^{\alpha,\beta,N}(x)$ the functions
\begin{align}\label{defom}
\Omega _\F ^{\alpha, \beta, N}(x)&= \frac{ \left|
  \begin{array}{@{}c@{}lccc@{}c@{}}
    & &&\hspace{-1.3cm}{}_{j=1,\ldots , k} \\
    \dosfilas{ h_{f}^{\alpha ,\beta ,N}(x+j-1) }{f\in F_1} \\
    \dosfilas{ s_{k-j+1}^{N-k+1}(x)s_{j-1}^{N+\beta+1}(x+j-1)h_{f}^{\alpha ,-\beta ,\beta+N}(x+j-1) }{f\in F_2}
  \end{array}
  \hspace{-.4cm}\right|}{\prod_{i=1}^{k_2}(N-x-k+1)_{k_2-i+1}(N+\beta-x-i+2)_{i-1}},\\
\nonumber
\Lambda _\F^{\alpha, \beta, N}(x)&= \frac{ \left|
  \begin{array}{@{}c@{}lccc@{}c@{}}
    & &&\hspace{-1.3cm}{}_{j=1,\cdots , k-1,k+1} \\
    \dosfilas{ h_{f}^{\alpha ,\beta ,N}(x+j-1) }{f\in F_1} \\
    \dosfilas{ s_{k-j+1}^{N-k+1}(x)s_{j-1}^{N+\beta+1}(x+j-1)h_{f}^{\alpha ,-\beta ,\beta+N}(x+j-1) }{f\in F_2}
  \end{array}
  \hspace{-.4cm}\right|}{\prod_{i=1}^{k_2}(N-x-k+1)_{k_2-i+1}(N+\beta-x-i+2)_{i-1}}.
\end{align}
To simplify the notation we sometimes write $\Omega_\F=\Omega_\F^{\alpha, \beta, N}$, $\Lambda_\F=\Lambda_\F^{\alpha, \beta, N}$.
Using Lemma \ref{lgp1}, we deduce that $\Omega _\F$  is always a polynomial of degree
$u_\F+k_1$.  Moreover, the leading coefficient of $\Omega _\F$ is
$$
V_{F_1}V_{F_2}\prod_{f_1\in F_1}r_{f_1}^{\alpha,\beta}\prod_{f_2\in F2}r_{f_2}^{\alpha,-\beta}\prod_{f_1\in F_1,f_2\in F_2}(\beta +f_1-f_2),
$$
where $V_F$ is the Vandermonde determinant (\ref{defvdm}) and $r_i^{a,b}=\frac{(a+b+i+1)_{i}}{(a+1)_{i}}$.

In a similar way, one can see that except for $F_2=\emptyset$, $\Lambda_\F$ is not a polynomial but $(N-x-k+1)\Lambda_\F$ is always a polynomial of degree  $u_\F+k_1+1$.

As for $h_n^\F$ (see (\ref{defmexa})), we have for $\Omega_\F$ the following alternative definition
\begin{equation*}\label{defoma}
\Omega _\F(x)= \frac{ \left|
  \begin{array}{@{}c@{}lccc@{}c@{}}
    &&&\hspace{-.9cm}{}_{1\le j\le k} \\
    \dosfilas{(f-j+2,f+\alpha+\beta+1)_{j-1} h_{f-j+1}^{\alpha +j-1,\beta +j-1,N-j+1}(x)}{f\in F_1} \\
    \dosfilas{ (\alpha +f+1,\beta -f)_{j-1}s_{k-j+1}^{N-k+1}(x)h_{f}^{\alpha +j-1,-\beta -j+1 ,\beta+N}(x) }{f\in F_2}
  \end{array}
  \hspace{-.3cm}\right|}{\prod_{i=0}^{k-1}(\alpha+1)_i\prod_{i=1}^{k_2}(N-x-k+1)_{k_2-i+1}(N+\beta-x-i+2)_{i-1}}
\end{equation*}

We also need the determinants $\Phi_n^\F$ and $\Psi_n^\F$, $n=0,\ldots, N-k$, defined by
\begin{align}\label{defphme}
\Phi^\F_n&=\left|
  \begin{array}{@{}c@{}lccc@{}c@{}}
    & &&\hspace{-1.3cm}{}_{j=1,\cdots , k} \\
    \dosfilas{R_{n+j-1}^{\alpha ,\beta ,N}(\lambda^{\alpha,\beta}(f))}{f\in F_1} \\
    \dosfilas{R_{n+j-1}^{\alpha ,-\beta ,\beta+N}(\lambda^{\alpha,-\beta}(f)) }{f\in F_2}
  \end{array}
  \hspace{-.3cm}\right|,\\\label{defpsme}
\Psi_n^F&=\left|
  \begin{array}{@{}c@{}lccc@{}c@{}}
    & &&\hspace{-1.3cm}{}_{j=1,\cdots , k-1, k+1} \\
    \dosfilas{R_{n+j-1}^{\alpha ,\beta ,N}(\lambda^{\alpha,\beta}(f))}{f\in F_1} \\
    \dosfilas{R_{n+j-1}^{\alpha ,-\beta ,\beta+N}(\lambda^{\alpha,-\beta}(f)) }{f\in F_2}
  \end{array}
  \hspace{-.3cm}\right|.
\end{align}
Using the duality (\ref{sdm2b}), we have
\begin{align}\label{duomph}
\xi _n\Omega _\F(n)&=(-1)^{k_2}\kappa(-N)_{n+k_1}\Phi_n^\F , \\\label{dulaps}
\xi _n(N-n-k+1)\Lambda _\F(n)&=(-1)^{{k_2}+1}\kappa(-N)_{n+k_1}\Psi_n^\F .
\end{align}

\begin{theorem}\label{th3.3} Let $\F=(F_1,F_2)$ be a pair of finite sets of positive integers. Then the polynomials $h_n^\F$ (\ref{defmex}), $n\in \sigma _\F$, are common eigenfunctions of the second order difference operator
\begin{equation}\label{sodomex}
D_\F=h_{-1}(x)\Sh_{-1}+h_0(x)\Sh_0+h_1(x)\Sh_{1},
\end{equation}
where
\begin{align}\label{jpm1}
h_{-1}(x)&=\frac{x(x-\beta-N-1+k_2)\Omega_\F(x+1)}{\Omega_\F(x)},\\\label{jpm2}
h_0(x)&=-(x+k)(x-\beta-N-1+k)-(x+\alpha+1+k)(x-N+k)\\\nonumber &\quad \quad +\Delta\left(\frac{(x+\alpha+k)(x-N-1+k))\Lambda_\F(x)}{\Omega_\F(x)}\right),\\\label{jpm3}
h_1(x)&=\frac{(x+\alpha+k+1)(x-N+k_1)\Omega_\F(x)}{\Omega_\F(x+1)},
\end{align}
and $\Delta $ denotes the first order difference operator $\Delta f=f(x+1)-f(x)$. Moreover $D_\F(h_n^\F)=\lambda(n-u_\F)h_n^\F$, $n\in \sigma_\F$.
\end{theorem}

\begin{proof}
The proof is similar to that of Theorem 3.3 in \cite{duch} but using here the three term recurrence relation for the dual Hahn  polynomials
(\ref{Mxttrr}) and the dualities (\ref{duaqnrn}), (\ref{duomph}) and (\ref{dulaps}).
\end{proof}

\bigskip

The determinantal definition (\ref{defmex}) of the polynomials $h_n^\F$, $n\in \sigma_\F$, automatically implies a factorization for the corresponding difference operator $D_\F$ (\ref{sodomex}) in two difference operators of order $1$. This is a consequence of the Sylvester identity (see \cite{Gant}, pp. 32, or \cite{duch}, Lemma 2.1). This can be done by choosing one of the components of $\F=(F_1,F_2)$ and removing one element in the chosen component. An iteration of this procedure shows
that the polynomials $h_n^\F$, $n\in \sigma_\F$, and the corresponding difference operator $D_\F$  can be constructed by applying a sequence of at most $k$ Darboux transform (see Definition 2.1 in \cite{dume}) to the Hahn system (where $k$ is the number of elements of $\F$). We display the details in the following lemma, where we remove one element of the component $F_2$ of $\F$, and hence we have to assume $F_2\not =\emptyset$. A similar result can be proved by removing one element of the component $F_1$. The proof proceeds in the same way as the proof of Lemma 3.6 in \cite{duch} and it is omitted.

\begin{lemma}\label{lfe} Let $\F=(F_1,F_2)$ be a pair of finite sets of positive integers and assume $F_2\not =\emptyset$. We define the first order difference operators $A_\F$ and $B_\F$ as
\begin{align}
A_\F&=\frac{(-x+\beta+N-k_2+1)\Omega _\F(x+1)}{\Omega_{\F_{2,\{ k_2\} }}(x+1)}\Sh_0-\frac{(-x+N-k_1)\Omega _\F(x)}{\Omega_{\F_{2,\{ k_2\} }}(x+1)}\Sh_1,
\\
B_\F&=\frac{-x\Omega _{\F_{2,\{ k_2\} }}(x+1)}{\Omega_{\F}(x)}\Sh_ {-1}+\frac{(x+\alpha+k)\Omega _{\F_{2,\{ k_2\} }}(x)}{\Omega_{\F}(x)}\Sh_0,
\end{align}
where $k_2$ is the number of elements of $F_2$ and the pair $\F_{2,\{ k_2\} }$ is defined by (\ref{deff2}).
Then $h_{n+u_\F}^\F(x)=A_\F(h_{n+u_{\F_{2,\{ k_2\}}}}^{\F_{2,\{ k_2\} }})(x)$, $n\not \in F_1$. Moreover
\begin{align*}
D_{\F_{2,\{ k_2\} }}&=B_\F A_\F-(\alpha+f_{k_2}^{2\rceil}+1)(\beta-f_{k_2}^{2\rceil})Id,\\
D_{\F}&=A_\F B_\F-(\alpha+f_{k_2}^{2\rceil}+1)(\beta-f_{k_2}^{2\rceil})Id.
\end{align*}
In other words, the system $(D_F,(h_n^\F)_{n\in \sigma _\F})$ can be obtained by applying a Darboux transform to the system
$(D_{\F_{2,\{ k_2\} }},(h_n^{\F_{2,\{ k_2\} }})_{n\in \sigma _{\F_{2,\{ k_2\} }}})$.
\end{lemma}

Analogous factorization can be obtained by removing instead of $f_{k_2}^{2\rceil}$ any other element $f_i^{2\rceil}$ of $F_2$, $1\le i<k_2$.

\section{Exceptional Hahn polynomials}
Under certain assumptions on the parameters $\alpha,\beta, N$, in the previous Section we have associated to each pair $\F =(F_1, F_2)$ of finite sets of positive integers the polynomials $h_n^{\alpha,\beta,N;\F}$, $n\in \sigma_\F$, which are always eigenfunctions of a second order difference operator with rational coefficients.
We are interested in the cases when, in addition, those polynomials are orthogonal and complete with respect to a positive measure.

As it happens with the Hahn family, for the existence of such a positive measure the parameter $N$ must be taken to be a positive integer and there will be only  a finite number of exceptional polynomials with non-null norm in each family. More precisely, for a positive integer $N$
and a pair $\F=(F_1,F_2)$ with $F_1\varsubsetneq\{1,\cdots, N\}$, write
\begin{equation}\label{defsigN}
\sigma_{N;\F} =\sigma_\F\cap \{u_\F,\cdots , N+u_\F\}.
\end{equation}
Notice that $\sigma_{N;\F}$ has exactly $N-k_1+1$ elements. Hence, along this section we will assume $N$ to be a positive integer
and $F_1\varsubsetneq\{1,\cdots, N\}$, as well as
\begin{equation}\label{cpar21}
\alpha, \beta=-1,-2,\cdots ,-N,\quad \alpha +\beta=-1,-2,\cdots ,-2N-1.
\end{equation}
In addition, when $F_2\not =\emptyset$, we also assume
\begin{equation}\label{cpar22}
\alpha-\beta\not =-1,-2,\cdots ,-N, \quad \beta \not=f^2_{k_2},f^2_{k_2}-1,\cdots, 0.
\end{equation}
Notice that the assumptions (\ref{cpar21}) and (\ref{cpar22}) imply the condition (\ref{cpara}) that we assume to define the Hahn admissibility in Section 2.3.

As for the Hahn family, when $N$ is a positive integer we only consider the polynomials $h_n^{\alpha,\beta,N;\F}$ (see (\ref{defmex})) for $n\in \sigma_{N;\F}$. Indeed, when $n\in \sigma_\F$ and $n> N+u_\F$, the polynomial $h_n^{\alpha,\beta,N;\F}$ is divisible by $(-x)_{N-k_1+1}$. Since the orthogonalizing measure $\omega^\F_{\alpha,\beta, N}$ for the exceptional Hahn polynomials will have support in $\{ 0,1,\cdots, N-k_1\}$ (see (\ref{momex})), we have that the polynomial $h_n^{\alpha,\beta,N;\F}$, $n\in \sigma_\F$ and $n> N+u_\F$, vanishes in the support of $\omega^\F_{\alpha,\beta, N}$ and then it is useless regarding the space $L^2(\omega^\F_{\alpha,\beta, N})$.

\begin{definition} The polynomials $h_n^{\alpha,\beta,N;\F}$, $n\in \sigma_{N;\F}$, defined by (\ref{defmex}) are called exceptional Hahn polynomials, if they are orthogonal and complete with respect to a positive measure.
\end{definition}

As we point out in Section \ref{sectadm}, the key concept for the construction of exceptional Hahn polynomials is the
Hahn admissibility (see Definition \ref{dadh}). Hahn admissibility  can also be characterized in terms of the measure $\rho_{\alpha,\beta,N}^\F$ (\ref{mraf}) and the sign of the polynomial $\Omega _\F ^{\alpha,\beta, N}(x)$ in $\{0,\cdots, N-k_1\}$.

\begin{lemma}\label{l3.1} Given a positive integer $N$, real numbers $\alpha, \beta$, satisfying (\ref{cpar21}) and (\ref{cpar22}) and a pair $\F =(F_1,F_2)$ of finite sets of positive integers with $F_1\varsubsetneq\{1,\cdots, N\}$, the following conditions are equivalent.
\begin{enumerate}
\item The measure $\rho_{\alpha, \beta,N}^\F$ (\ref{mraf}) is either positive or negative.
\item $\alpha, \beta, N$ and $\F$ are Hahn admissible.
\item $\displaystyle \frac{\Omega_\F ^{\alpha,\beta, N}(n)\Omega_\F ^{\alpha,\beta, N}(n+1)}{\binom{\alpha+k+n}{n}\binom{\beta+N-k_2-n}{N-k_1-n}}$ is non-null and has constant sign for $n=0,\cdots, N-k_1$, where the polynomial $\Omega_\F^{\alpha,\beta , N}$ is defined by (\ref{defom}).
\end{enumerate}
Moreover, if any of these conditions hold, then
\begin{enumerate}
\item[i.] the sign of the measure $\rho_{\alpha, \beta,N}^\F$ is the sign of the function $\HH^{\alpha, \beta,N}_\F $ (\ref{defadmh});
\item[ii.] the constant sign mentioned in part 3 is equal to the sign of $(\alpha+1)_k(\beta+1)_{k_1-k_2}\HH^{\alpha, \beta,N}_\F $;
\item[iii.] if $N>\widehat{\alpha+\beta+1}$, then the constant sign mentioned in part 3 is positive.
\end{enumerate}
\end{lemma}

\begin{proof}
Since
\begin{equation}\label{pmc}
\lambda(u)-\lambda(v)=(u-v)(u+v+\alpha+\beta+1),
\end{equation}
the equivalence between parts 1 and 2 (as well as part i) is an easy consequence of the definitions of Hahn admissibility (\ref{defadmh}) and of the measure $\rho _{\alpha,\beta, N}^\F$.

We now prove the equivalence between part 1 and part 3.

Part 1 $\Rightarrow$ part 3. We assume that the measure $\rho_{\alpha, \beta, N}^\F$ is positive (if it is negative, we proceed in a similar way).
Write $\rho_{\alpha, \beta, N}^\F (x)$ for the mass of the discrete measure $\rho_{\alpha, \beta, N}^\F$ at the point $\lambda (x-u_\F)$. Using (\ref{pmc}), we can write
\begin{align*}
\rho_{\alpha, \beta, N}^\F(x)=\prod_{f\in F_1}&(x-u_\F-f)(x-u_\F+f+\alpha+\beta+1)\\ & \times\prod_{f\in F_2}(x-u_\F-f+\beta)(x-u_\F+f+\alpha+1)
w_{*;\alpha,\beta,N}(x-u_\F).
\end{align*}
The assumptions (\ref{cpar21}) and (\ref{cpar22}) on the parameters $\alpha,\beta, N$ and the finite set $F_1$ imply that $\rho_{\alpha, \beta, N}^\F(x)\not =0$, $x=u_\F ,\cdots , u_\F+N$, except when $x=u_\F+f$, $f\in F_1$. This means that the positive measure $\rho_{\alpha, \beta, N}^\F$ is supported in the finite set of nonnegative integers $\{\lambda(x-u_\F),x\in \sigma_{N;\F}\}$. Notice that this finite set has exactly $N-k_1+1$ points. Since the polynomials $q_n^\F$ (see (\ref{defqnme})), $n=0,\cdots , N-k_1$, are orthogonal with respect to the measure $\rho_{\alpha, \beta, N}^\F$ and the degree of $q_n^\F$ is at most $N-k_1$, they have positive $L^2$-norm. According to (\ref{n2q}), we have
\begin{equation}\label{nssu}
\langle q_n^\F,q_n^\F\rangle =\frac{(-1)^k (\alpha +1)_n}{(\alpha+1)_{n+k}}\langle R_n^{\alpha,\beta, N},R_n^{\alpha,\beta, N}\rangle \Phi_n^\F\Phi_{n+1}^\F=
\frac{(-1)^k (-N)_n^2(\alpha +1)_n\Phi_n^\F\Phi_{n+1}^\F}{(\alpha+1)_{n+k}\binom{\alpha+n}{n}\binom{\beta+N-n}{N-n}}.
\end{equation}
Using the duality (\ref{duomph}), we get
$$
\Phi_n^\F\Phi_{n+1}^\F=\frac{\xi_n\xi_{n+1}}{\kappa ^2(-N)_{n+k_1}(-N)_{n+1+k_1}}\Omega_\F(n)\Omega_\F(n+1).
$$
Using the definition of $\xi_n$ in Lemma \ref{lem3.2}, and after an easy computation, we have for $n=0,\cdots , N-k_1$,
$$
\Phi_n^\F\Phi_{n+1}^\F=A^2(-1)^k(N-n-k_1+1)_{k_1}(N+\beta-n-k_2+1)_{k_2}\Omega_\F(n)\Omega_\F(n+1),
$$
where $A$ is certain real number (depending on $n$). Using (\ref{nssu}), we deduce  for $n=0,\cdots , N-k_1$, that
$$
\frac{(\alpha+1)_n(N+\beta-n-k_2+1)_{k_2}\Omega_\F(n)\Omega_\F(n+1)}{(\alpha+1)_{n+k}\binom{\alpha+n}{n}\binom{\beta+N-n}{N-n}}> 0.
$$
This can be rewritten as
\begin{equation}\label{nssu1}
\frac{(\alpha+1)_n(N+\beta-n-k_2+1)_{k_2}\binom{\alpha+k+n}{n}\binom{\beta+N-n-k_2}{N-n-k_1}}{(\alpha+1)_{n+k}\binom{\alpha+n}{n}\binom{\beta+N-n}{N-n}}
\frac{\Omega_\F(n)\Omega_\F(n+1)}{\binom{\alpha+k+n}{n}\binom{\beta+N-n-k_2}{N-n-k_1}}>0.
\end{equation}
A simple computation gives
$$
\frac{(\alpha+1)_n(N+\beta-n-k_2+1)_{k_2}\binom{\alpha+k+n}{n}\binom{\beta+N-n-k_2}{N-n-k_1}}{(\alpha+1)_{n+k}\binom{\alpha+n}{n}\binom{\beta+N-n}{N-n}}
=\frac{(N-n)!(\beta+1)_{k_1-k_2}}{(N-k_1-n)!(\alpha+1)_{k}},
$$
which is non-null and has constant sign for $n=0,\cdots , N-k_1$.

We then conclude from (\ref{nssu1}) that if $\rho_{\alpha, \beta, N}^\F$ is a positive measure, then
\begin{equation}\label{lop}
(\alpha+1)_{k}(\beta+1)_{k_1-k_2}\frac{\Omega_\F ^{\alpha,\beta, N}(n)\Omega_\F ^{\alpha,\beta, N}(n+1)}{\binom{\alpha+k+n}{n}\binom{\beta+N-k_2-n}{N-k_1-n}}
\end{equation}
is positive for $n=0,\cdots, N-k_1$. In a similar form, we can prove that if $\rho_{\alpha, \beta, N}^\F$ is a negative measure, then
(\ref{lop}) is negative for $n=0,\cdots, N-k_1$. This proves part 3. Using part i, part ii also follows easily.

Part 3 $\Rightarrow$ part 1. Assume that the sign mentioned in part 3 is positive (if it is negative we can proceed in a similar way). Using the duality (\ref{duomph}), the definition of
$\xi_n$ in Lemma \ref{lem3.2} and proceeding as before, we conclude that the polynomial $q_n^\F$, $n=0,\cdots , N-k_1$,
are orthogonal with respect to $\rho_{\alpha,\beta, N}^\F$ and its $L^2$-norm is non-null and has constant sign. Since the degree of $q_n^\F$ is $n$, $n=0,\cdots, N-k_1$, and the measure $\rho_{\alpha,\beta, N}^\F$ is supported in $N-k_1+1$ points, part 1 of Lemma \ref{ldmp} gives that $\rho_{\alpha,\beta, N}^\F$ is either a positive or a negative measure.

If $N>\widehat{\alpha+\beta+1}$, part 1 of Lemma \ref{ladm} shows that the sign of $(\beta +1)_{k_1-k_2}(\alpha+1)_k$ is equal to the sign of
$\HH ^{\alpha,\beta, N}_\F$. Part ii then implies that the sign mentioned in part 3 has to be positive.
\end{proof}

Let us remark that we have used part 1 of Lemma \ref{ladm} (still to be proved) only in the proof of part iii of the previous Lemma.

According to Lemma \ref{l3.1} and part 1 of Lemma \ref{ladm}, if $\alpha,\beta, N$ and $\F$ are Hahn admissible, we have $\alpha+k>-1$, $\beta+k_1-k_2> -1$ and $\displaystyle \frac{\Omega_\F ^{\alpha,\beta, N}(n)\Omega_\F ^{\alpha,\beta, N}(n+1)}{\binom{\alpha+k+n}{n}\binom{\beta+N-k_2-n}{N-k_1-n}}>0$, for all $n=0,\cdots , N-k_1$. One can then deduce that if $\alpha,\beta, N$ and $\F$ are admissible, then $\Omega_\F ^{\alpha,\beta, N}(n)\Omega_\F ^{\alpha,\beta, N}(n+1)>0$, for all $n\in \NN$. We point out that the converse is  not true. Indeed, take $\alpha=-7/2$, $\beta=9$, $N=20$ and $F_1=\{1\}$, $F_2=\emptyset$.
A straightforward computation gives
$$
\Omega_\F^{\alpha,\beta,N} (n)\Omega_\F^{\alpha,\beta,N} (n+1)=(3n+20)(3n+23)>0,\quad n=0,\cdots , N.
$$
However, it is easy to see that $\alpha,\beta, N$ and $\F$ are not admissible ((\ref{defadmh}) is nonnegative for $x=0,1,2,4,\cdots , N,$ but negative for $x=3$).

\bigskip

In the following Theorem we prove that when  $\alpha,\beta ,N$ and $\F$ are Hahn admissible the polynomials $h_n^{\alpha,\beta,N;\F}$, $n\in \sigma _{N;\F}$, are orthogonal and complete with respect to a constant sign measure.

\begin{theorem}\label{th4.5} Given a positive integer $N$, real numbers $\alpha,\beta$ satisfying (\ref{cpar21}) and (\ref{cpar22})  and a pair $\F =(F_1,F_2)$ of finite sets of positive integers with $F_1 \varsubsetneq \{1,\cdots, N\}$, assume that $\alpha, \beta, N$ and $\F$ are admissible. Then
the  polynomials $h_n^{\alpha,\beta, N;\F}$, $n\in \sigma _{N;\F}$, are orthogonal and complete in $L^2(\omega_{\alpha,\beta, N}^{\F})$, where $\omega_{\alpha,\beta, N}^{\F}$ is the measure (which it is either positive or negative)
\begin{equation}\label{momex}
\omega_{\alpha,\beta,N}^\F=\sum_{x=0}^{N-k_1} \frac{\binom{\alpha +k+x}{x}\binom{\beta +N-k_2-x}{N-k_1-x}}{\Omega_\F^{\alpha\beta,N}(x)\Omega_\F^{\alpha\beta,N}(x+1)}\delta_x.
\end{equation}
Hence $h_n^{\alpha,\beta, N;\F}$, $n\in \sigma _{N;\F}$, are exceptional Hahn polynomials. Moreover, if $N>\widehat{\alpha+\beta+1}$ then the measure
$\omega_{\alpha,\beta,N}^\F$ is positive.
\end{theorem}

\begin{proof}
The proof that $h_n^{\alpha,\beta, N;\F}$, $n\in \sigma _{N;\F}$, are orthogonal polynomials with respect to the  measure (\ref{momex}) can be proved as in Th. 4.4 in \cite{duch} or Th. 4.3 in \cite{dume}.

We now prove that they are complete in $L^2(\omega_{\alpha,\beta, N}^{\F})$.
Using Lemma \ref{l3.1} and taking into account that $\alpha,\beta, N$ and $\F$ are admissible, it follows that the measure $\rho _{\alpha,\beta, N}^{\F}$ (\ref{mraf}) is either positive or negative and it is supported at $N-k_1+1$ points. We assume that it is a positive measure. We also have that the polynomials $q_n^\F$ (where $q_n^\F$ is defined by (\ref{defqnme})), $n=0,\cdots, N-k_1$, have degree $n$ and positive $L^2$-norm. Using part 2 of Lemma \ref{ldmp}, we deduce  that the finite sequence $q_n^\F/\Vert q_n^\F\Vert _2$, $n=0,\cdots, N-k_1$, is an orthonormal basis in $L^2(\rho_{\alpha,\beta, N}^\F)$.

For $s\in \sigma _\F$, consider the function $h_s(x)=\begin{cases} 1/\rho _{\alpha,\beta, N}^{\F}(s),& x=\lambda(s-u_\F),\\ 0,& x\not =\lambda(s-u_\F), \end{cases}$ where by $\rho _{\alpha,\beta, N}^{\F}(s)$ we denote the mass  of the discrete measure $\rho_{\alpha,\beta, N}^\F$ at the point $\lambda(s-u_\F)$. Since the support of $\rho _{\alpha,\beta, N}^{\F}$ is $\{\lambda(y-u_\F), y\in\sigma_{N;\F}\}$, we get that $h_s\in L^2(\rho_{\alpha,\beta, N}^\F)$. Its Fourier coefficients with respect to the orthonormal basis $(q_n^\F/\Vert q_n^\F\Vert _2)_n$ are $q_n^\F(\lambda(s-u_\F))/\Vert q_n^\F\Vert _2$, $n=0,\cdots, N-k_1$. Hence
\begin{equation}\label{pf1}
\sum _{n=0}^{N-k_1} \frac{q_n^\F(\lambda(s-u_\F))q_n^\F(\lambda(r-u_\F))}{\Vert q_n^\F\Vert _2 ^2}=\langle h_s,h_r\rangle _{\rho_{a,c}^\F}=\frac{1}{\rho_{\alpha,\beta, N}^\F(s)}\delta_{s,r}.
\end{equation}
This is the dual orthogonality associated to the orthogonality
$$
\sum_{u\in \sigma _\F}q_n^\F(\lambda(u-u_\F))q_m^\F(\lambda(u-u_\F))\rho _{\alpha,\beta, N}^{\F}(u)=\langle q_n^\F,q_n^\F\rangle \delta_{n,m}
$$
of the polynomials $q_n^\F$, $n=0,\cdots, N-k_1$, with respect to the positive measure $\rho _{\alpha,\beta, N}^{\F}$ (see, for instance, \cite{At}, Appendix III, or \cite{KLS}, Th. 3.8).

Using (\ref{nssu}), (\ref{normedh}) and the duality (\ref{duomph}), we get
\begin{equation}\label{pf0}
\frac{1}{\langle q_n^\F,q_n^\F\rangle _{\rho_{\alpha,\beta,N}^\F}}=\omega_{\alpha,\beta,N}^{\F}(n)x_n,
\end{equation}
where $x_n$ is the non-null real number given by
\begin{equation}\label{pf2}
x_n=\frac{(-1)^k\kappa^2(-N+n+k_1)(-N+n)_{k_1}^2(\alpha +1)_k\binom{\beta+N-n}{N-n}}{\xi_n\xi_{n+1}\binom{\beta+N-n-k_2}{N-n-k_1}},
\end{equation}
and $\kappa$ and $\xi_n$ are defined in Lemma \ref{lem3.2}.

Using now the duality (\ref{duaqnrn}), we can rewrite (\ref{pf1}) for $s=r$ as
\begin{equation}\label{pf3}
\sum _{n=0}^{N-k_1} \omega_{\alpha,\beta, N}^{\F}(n)(h_r^\F(n))^2\frac{x_n\xi_n^2}{\kappa^2\zeta_r^2}=\frac{1}{\rho_{\alpha,\beta, N}^\F(r)}.
\end{equation}
A straightforward computation using (\ref{pf2}) and the definitions of $\kappa$, $\xi_n$ and $\zeta_r$ in Lemma \ref{lem3.2} gives
\begin{equation}\label{pf4}
\frac{x_n\xi_n^2}{\kappa^2\zeta_r^2}=\frac{(\alpha+1)_k(\beta+1)_{k_1-k_2}w_{*;\alpha,\beta,N}^2(r-u_\F)}{(-N)^2_{r-u_\F}(\rho_{\alpha,\beta, N}^\F(r))^2},
\end{equation}
where $w_{*;\alpha,\beta,N}(x)$ is the mass of the dual Hahn measure at $\lambda (x)$ given by (\ref{masdh}).

Inserting (\ref{pf4}) in (\ref{pf3}), we get
\begin{equation}\label{pf5}
\langle h_r^\F,h_r^\F\rangle_{\omega_{\alpha,\beta, N}^{\F}}=\frac{(-N)^2_{r-u_\F}}{(\alpha+1)_k(\beta+1)_{k_1-k_2}w_{*;\alpha,\beta,N}^2(r-u_\F)}\rho_{\alpha,\beta, N}^\F(r).
\end{equation}
This shows that the orthogonal polynomials $h_r^\F$, $r\in \sigma_{N;\F}$, have non-null $L^2$ norm with constant sign for $r\in \sigma_{N;\F}$. On the other hand, Lemma 4.2 shows that the measure $\omega_{\alpha,\beta, N}^{\F}$ is either positive or negative and has $N-k_1+1$ point in its support
(moreover, this measure is positive when $N>\widehat{\alpha+\beta+1}$). Since we have $N-k_1+1$ polynomials $h_r^\F$ of degree $r$, we can conclude using part 2 of Lemma \ref{ldmp} that they form an orthogonal basis in $L^2(\omega_{\alpha,\beta, N}^{\F})$.
\end{proof}

\section{Constructing polynomials which are eigenfunctions of second order differential operators}
One can construct exceptional Jacobi polynomials by taking limit in the exceptional Hahn polynomials. We use the basic limit
(\ref{blmel}).

In the next two sections we assume the same constrains
(\ref{cpar11}) and (\ref{cpar12}) on the parameters $\alpha, \beta $ as in Section 3.
These assumptions are needed to define the polynomial $P_n^{\alpha, \beta;\F}$ (see (\ref{deflax}) below) and to guarantee
that it has degree $n$
(actually this last condition can be guaranteed  using the weaker assumption $\beta \not \in (F_2- F_1)\cup \cup_{n\in \sigma_\F}(-n+u_\F+F_2)$).

Given a pair $\F=(F_1,F_2)$ of finite sets of positive integers, using the expression (\ref{defmexa}) for the polynomials $h_n^{\alpha,\beta,N;\F}$, $n\in\sigma_\F$, setting $x\to (1-x)N/2$ and taking limit as $N\to +\infty$, we get (up to normalization constants) the polynomials, $n\in \sigma _\F$,
\begin{equation}\label{deflax}
P_n^{\alpha,\beta ;\F}(x)= \frac{ \left|
  \begin{array}{@{}c@{}lccc@{}c@{}}
    & (-1)^{j-1}(P_{n-u_\F}^{\alpha ,\beta })^{(j-1)}(x) &&\hspace{-.9cm}{}_{1\le j\le k+1} \\
    \dosfilas{ (-1)^{j-1}(P_{f}^{\alpha ,\beta })^{(j-1)}(x) }{f\in F_1} \\
    \dosfilas{ (\beta-f)_{j-1}(1+x)^{k-j+1}P_{f}^{\alpha +j-1 ,-\beta-j+1 }(x)}{f\in F_2}
  \end{array}
  \hspace{-.4cm}\right|}{(1+x)^{k_2(k_2-1)}}.
\end{equation}
More precisely
\begin{equation}\label{lim1}
\lim_{N\to +\infty}\frac{h_n^{\alpha,\beta,N;\F}\left((1-x)N/2\right)}{N^n}=\upsilon_n^{\alpha;\F}P_n^{\alpha,\beta ;\F}(x)
\end{equation}
uniformly in compact sets, where
\begin{equation}\label{defupn}
\upsilon_n^{\alpha;\F}=\frac{(-1)^{n}(-2)^{\binom{k_1+1}{2}+\binom{k_2}{2}}(n-u_\F)!\prod_{f\in F_1,F_2}f!}{(\alpha+1)_{n-u_\F}\prod_{f\in F_1,F_2}(\alpha+1)_f}.
\end{equation}

Notice that $P_n^{\alpha,\beta ;\F}$ is a polynomial of degree $n$ with leading coefficient equal to
$$
\frac{V_{F_1}V_{F_2}\prod_{i\in \{n-u_\F\},F_1,F_2}(\alpha+\epsilon_i\beta+i+1)_i\prod_{i\in \{n-u_\F\},F_1,f_2\in F_2}(\beta +i-f_2)\prod_{f\in F_1}(f-n+u_\F)}{(-1)^{\binom{k_1+1}{2}+\binom{k_2}{2}}2^{n+\binom{k_1+1}{2}+\binom{k_2}{2}}(n-u_\F)!\prod_{f\in F_1,F_2}f!},
$$
where $V_F$ is the Vandermonde determinant defined by (\ref{defvdm}) and $\epsilon_i=1$, for $i\in \{n-u_\F\},F_1$ and
$\epsilon_i=-1$, for $i\in F_2$.

We introduce the associated polynomial
\begin{equation}\label{defhom}
\Omega _{\F}^{\alpha,\beta}(x)=\frac{ \left|
  \begin{array}{@{}c@{}lccc@{}c@{}}
    &  &&\hspace{-.9cm}{}_{1\le j\le k} \\
    \dosfilas{ (-1)^{j-1}(P_{f}^{\alpha ,\beta })^{(j-1)}(x) }{f\in F_1} \\
    \dosfilas{ (\beta-f)_{j-1}(1+x)^{k-j}P_{f}^{\alpha +j-1 ,-\beta-j+1 }(x)}{f\in F_2}
  \end{array}
  \hspace{-.4cm}\right|}{(1+x)^{k_2(k_2-1)}}.
\end{equation}
Notice that $\Omega_{\F}^{\alpha,\beta}$ is a polynomials of degree $u_\F+k_1$. To simplify the notation we sometimes write $\Omega_\F=\Omega_\F^{\alpha,\beta }$.

We straightforwardly have
\begin{equation}\label{rromh}
P_{u_F}^{\alpha,\beta ;\F}(x)=z_{\alpha,\beta;\F}\Omega_{\F_\Downarrow }^{\alpha +s_\F,\beta+s_\F}(x),
\end{equation}
where
$$
z_{\alpha,\beta;\F}=
\frac{\prod_{f\in F_1}(f+\alpha+\beta+1)_{\min (f,s_\F)}\prod_{f\in F_2}(\beta-f)_{s_\F}}{(-2)^{s_\F(2k_1-s_\F+1)/2}},
$$
and the positive integer $s_\F$ and the pair $\F_\Downarrow$ are defined by (\ref{defs0f}) and (\ref{deffd}), respectively.

We will need to know the value at $\pm 1$ of the polynomial $\Omega_\F^{\alpha,\beta }$.

\begin{lemma}\label{v0le}
Let $\F$ be a pair of finite sets of positive integers, then $\Omega_\F^{\alpha,\beta }(1)$ and $\Omega_\F^{\alpha,\beta }(-1)$ are polynomials in $\alpha$ and $\beta$ which do not vanish when $\alpha$ and $\beta$ satisfy (\ref{cpar11})  and (\ref{cpar12}).
Moreover, if we write $\epsilon_1=1, \epsilon_2=-1$ then
\begin{align}\label{v0l}
\Omega_\F^{\alpha,\beta }(1)&=
\frac{\prod_{j=1}^2 V_{F_j}\prod _{i=1}^{k_j}(\alpha +i)_{k_j-i+1}\prod _{f\in F_j}(\alpha +k_j+1)_{f-k_j}}{(-2)^{\binom{k_1}{2}+\binom{k_2}{2}}\prod _{f\in F_1, F_2}f!\prod _{i=1}^{\min\{k_1,k_2\}}(\alpha +i)_{k-2i+1}}\\\nonumber &\hspace{0.6cm}\times
\prod_{l=1}^2\prod_{1=i<j=k_l}(\alpha +\epsilon_l\beta +f^{l\rceil}_i+f^{l\rceil}_j+1)\\\nonumber&\hspace{0.6cm}\times\prod_{f\in F_1}\prod _{g\in F_2}(\alpha +f+g+1)(\beta +f-g).
\end{align}
\begin{align}\label{v0-l}
\Omega_\F^{\alpha,\beta }(-1)&=
\frac{\prod_{j=1}^2 V_{F_j}\prod _{i=1}^{k_j}(\epsilon_j\beta +i)_{k_j-i+1}\prod _{f\in F_j}(\epsilon_j\beta +k_j+1)_{f-k_j}}{(-1)^{\sum_{f\in F_1,F_2}f}(2)^{\binom{k_1}{2}+\binom{k_2}{2}}\prod _{f\in F_1,F_2}f!}\\\nonumber &\hspace{0.6cm}\times
\prod_{l=1}^2\prod_{1=i<j=k_l}(\alpha +\epsilon_l\beta +f^{l\rceil}_i+f^{l\rceil}_j+1)\prod_{j=1}^{\min(k_1,k_2)}\prod _{i=j-k_1}^{k_2-j}(\beta -i).
\end{align}
\end{lemma}

\begin{proof}
The proof follows by a carefully computation using that
$$
(P_n^{\alpha,\beta})^{(i)} (1)=\frac{i!\binom{n+\alpha}{n-i}\binom{n+\alpha+\beta+i}{i}}{2^i},\quad (P_n^{\alpha,\beta})^{(i)} (-1)=\frac{i!\binom{n+\beta}{n-i}\binom{n+\alpha+\beta+i}{i}}{(-1)^{n+i}2^i}
$$
and standard determinant techniques.
Because of the value above of the Jacobi polynomials and their derivatives at $\pm 1$, $\Omega_\F^{\alpha,\beta }(\pm 1)$ is clearly a polynomial in both $\alpha $ and $\beta$; one can also see that the right hand side of (\ref{v0l}) is a polynomial in $\alpha$ because each factor of the form $\alpha +s$ in the denominator cancels with one in the numerator. It is now easy to see that if $\alpha$ and $\beta$ satisfy (\ref{cpar11})  and (\ref{cpar12}),  the right hand side of (\ref{v0l}) and (\ref{v0-l}) do not vanish.

\end{proof}

Passing again to the limit, we can transform the second order difference operator (\ref{sodomex}) in a second order differential operator with respect to which the polynomials $P_n^{\alpha,\beta ;F}$, $n\in\sigma_\F$, are eigenfunctions.

\begin{theorem}\label{th5.1} Given real numbers $\alpha $ and $\beta $ satisfying (\ref{cpar11})  and (\ref{cpar12}), and a pair $\F$ of finite sets of positive integers, the polynomials $P_n^{\alpha,\beta;\F}$, $n\in \sigma _\F$,
are common eigenfunctions of the second order differential operator
\begin{align*}
D_F&=(1-x^2)\partial ^2+h_1(x)\partial+h_0(x),\qquad \partial=d/dx,\\\nonumber
h_1(x)&=\beta-\alpha -2k_2-(\alpha+\beta+2k_1+2)x-2(1-x^2)\frac{\Omega_\F'(x)}{\Omega_\F(x)},\\\nonumber
h_0(x)&=-\lambda(k_1) +[\alpha-\beta+2k_2+(2k_1+\alpha +\beta)x]\frac{\Omega_\F'(x)}{\Omega_\F(x)}+(1-x^2)\frac{\Omega_\F''(x)}{\Omega_\F(x)}.
\end{align*}
More precisely $D_\F(P_n^\F)=-\lambda(n- u_\F)P_n^\F(x)$.
\end{theorem}

\begin{proof}
We omit the proof because proceeds as that of Theorem 5.1 in \cite{duch} and using the basic limit (\ref{blmel}) and its consequences
\begin{align}\label{lim2}
\lim _{N\to +\infty}\frac{\Omega _\F^{\alpha,\beta,N}(x_N)}{N^{u_\F+k_1}}&=\upsilon^\alpha_{\F}\Omega _\F^{\alpha,\beta} (x),\\\label{lim3}
\lim _{N\to +\infty}\frac{\Omega _\F^{\alpha,\beta,N}(x_N+1)-\Omega _\F^{\alpha,\beta,N}(x_N)}{N^{u_\F+k_1-1}}&=-2\upsilon^\alpha_{\F}(\Omega _\F^{\alpha,\beta} )' (x),\\\nonumber
\lim _{N\to +\infty}\frac{\Omega _\F^{\alpha,\beta,N}(x_N+1)-2\Omega _\F^{\alpha,\beta,N}(x_N)+\Omega _\F^{\alpha,\beta,N}(x_N-1))}{N^{u_\F+k_1-2}}&=4\upsilon^\alpha_{\F}(\Omega _\F^{\alpha,\beta} )'' (x).
\end{align}
where   $x_N=(1-x)N/2$ and
\begin{equation}\label{defup}
\upsilon^\alpha_{\F}=\frac{(-1)^{u_\F+k_1}(-2)^{\binom{k_1}{2}+\binom{k_2}{2}}\prod_{f\in F_1,F_2}f!}{\prod_{f\in F_1,F_2}(\alpha+1)_f}.
\end{equation}
\end{proof}

We can factorize the second order differential operator $D_\F$ as product of two first order differential operators.
This can be done by choosing one of the components of $\F=(F_1,F_2)$ and removing one element in the chosen component. An iteration  shows that the system $(D_\F, (P_n^{\alpha,\beta; \F})_{n\in \sigma_\F})$ can be constructed by applying a sequence of $k$ Darboux transforms to the Jacobi system (see Definition 2.1 in \cite{dume}). We display the details in the following lemma (the proof is omitted because is analogous to the proof of
Lemma 5.2 in \cite{duch} or 5.4 in \cite{dume}).

\begin{lemma}\label{lfel} Let $\F=(F_1,F_2)$ be a pair of finite sets of positive integers.

\noindent
If $F_1\not =\emptyset$, we define the first order differential operators $A_{1,\F}$ and $B_{2,\F}$ as
\begin{align*}
A_{1,\F}&=\frac{\Omega _\F(x)}{\Omega_{\F_{1,\{ k_1\}}}(x)}\partial-\frac{\Omega _\F'(x)}{\Omega_{\F_{1, \{ k_1\}}}(x)},\\
B_{1,\F}&=\frac{(1-x^2)\Omega _{\F_{1,\{ k_1\}}}(x)}{\Omega_{\F}(x)}\partial-\frac{(1-x^2)\Omega '_{\F_{1,\{ k_1\}}}(x)+((\alpha-\beta+2k_2)+(\alpha+\beta+2k_1)x)\Omega_{\F_{1,\{ k_1\}}}(x)}{\Omega_{\F}(x)},
\end{align*}
where the pair $\F_{1,\{ k_1\} }$ is defined by (\ref{deff1}).
Then for $n\not \in \F_1$,
\begin{align*}
P_{n+u_\F}^{\alpha,\beta ;\F}(x)&=A_{1,\F}(P_{n+u_{\F_{1,\{ k_1\}}}}^{\alpha,\beta; \F_{1,\{ k_1\}}})(x),\\
B_{1,\F}(P_{n+u_\F}^{\alpha,\beta ;\F})(x)&=-(n-f^{1\rceil}_{k_1})(n+f^{1\rceil}_{k_1}+\alpha+\beta+1)P_{n+u_{\F_{1,\{ k_1\}}}}^{\alpha,\beta; \F_{1,\{ k_1\}}}(x).
\end{align*}
Moreover
\begin{align*}
D_{\F_{1,\{ k_1\}}}&=B_{1,\F} A_{1,\F}-\lambda(f^{1\rceil}_{k_1})Id,\\
D_{\F}&=A_{1,\F} B_{1,\F}-\lambda(f^{1\rceil}_{k_1})Id.
\end{align*}
If $F_2\not =\emptyset$, we define the first order differential operators $A_{2,\F}$ and $B_{2,\F}$ as
\begin{align*}
A_{2,\F}&=\frac{(1+x)\Omega _\F(x)}{\Omega_{\F_{2,\{ k_2\}}}(x)}\partial-\frac{(1+x)\Omega _\F'(x)-(\beta+k_1-k_2+1)\Omega _\F(x)}{\Omega_{\F_{2, \{ k_2\}}}(x)},\\
B_{2,\F}&=\frac{(1-x)\Omega _{\F_{2,\{ k_2\}}}(x)}{\Omega_{\F}(x)}\partial-\frac{(1-x)\Omega '_{\F_{2,\{ k_2\}}}(x)+(\alpha+k)\Omega_{\F_{2,\{ k_2\}}}(x)}{\Omega_{\F}(x)},
\end{align*}
where  the pair $\F_{2,\{ k_2\} }$ is defined by (\ref{deff2}).
Then for $n\not \in F_{1}$,
\begin{align*}
P_{n+u_\F}^{\alpha,\beta ;\F}(x)&=A_{2,\F}(P_{n+u_{\F_{2,\{ k_2\}}}}^{\alpha,\beta; \F_{2,\{ k_2\}}})(x),\\
B_{2,\F}(P_{n+u_\F}^{\alpha,\beta ;\F})(x)&=-(n+\alpha+f_{k_2}^{2\rceil}+1)(n+\beta-f_{k_2}^{2\rceil})P_{n+u_{\F_{2,\{ k_2\}}}}^{\alpha,\beta; \F_{2,\{ k_2\}}}(x),
\end{align*}
Moreover
\begin{align*}
D_{\F_{2,\{ k_2\}}}&=B_{2,\F} A_{2,\F}-\lambda(f_{k_2}^{2\rceil}-\beta)Id,\\
D_{\F}&=A_{2,\F} B_{2,\F}-\lambda(f_{k_2}^{2\rceil}-\beta)Id.
\end{align*}
\end{lemma}

\section{Exceptional Jacobi polynomials}
In the previous Section, assuming the constrains (\ref{cpar11}) and (\ref{cpar12}) on the parameters $\alpha,\beta $, we have associated to each pair $\F$ of finite sets of positive integers the polynomials $P_n^{\alpha,\beta ;\F}$, $n\in \sigma_\F$,
which are always eigenfunctions of a second order differential operator with rational coefficients.
We are interested in the cases when, in addition, those polynomials are orthogonal and complete with respect to a positive measure.

\begin{definition} The polynomials $P_n^{\alpha,\beta ;F}$, $n\in \sigma_\F$, defined by (\ref{deflax}) are called exceptional Jacobi polynomials, if they are orthogonal and complete with respect to a positive measure.
\end{definition}

In the following Theorem we construct Jacobi exceptional polynomials (the proof is similar to that of Theorem 6.3 in \cite{dume} and it is omitted).

\begin{theorem}\label{th6.3} Given a pair $\F$ of finite sets of positive integers and real numbers $\alpha$ and $\beta$ satisfying (\ref{cpar11}) and (\ref{cpar12}), assume that
\begin{equation}\label{superc}
\alpha+k+1>0,\quad \beta+k_1-k_2+1>0,\qquad \Omega_\F^{\alpha,\beta}(x)\not=0,\quad x\in [-1,1].
\end{equation}
Then the polynomials $P_n^{\alpha,\beta ;\F}$, $n\in \sigma _\F$,
 are orthogonal with respect to the positive weight
\begin{equation}\label{molax}
\omega_{\alpha,\beta;\F}(x)=\frac{(1-x)^{\alpha +k}(1+x)^{\beta+k_1-k_2}}{(\Omega_\F^{\alpha,\beta}(x))^2},\quad -1<x<1,
\end{equation}
and their linear combinations  are dense in  $L^2(\omega_{\alpha,\beta; \F})$. Hence $P_n^{\alpha,\beta ;\F}$, $n\in \sigma _\F$, are
exceptional Jacobi polynomials.
\end{theorem}

The Jacobi admissibility of $\alpha,\beta $ and $\F$ is a necessary condition for the assumptions (\ref{superc}) in the previous Theorem.

\begin{theorem}\label{alv} Given a pair $\F$ of finite sets of positive integers and real numbers $\alpha$ and $\beta$ satisfying (\ref{cpar11}) and (\ref{cpar12}), assume that (\ref{superc}) holds. Then $\alpha,\beta $ and $\F$ are Jacobi admissible.
\end{theorem}

\begin{proof}

Proceeding as in \cite{dupe}, the theorem is an easy consequence of the following complex orthogonality for the exceptional Jacobi polynomials.

Consider the path $\Lambda$ encircling the points $+1$ and $-1$ first in a positive sense and then in a negative sense, as shown in Fig. 2.1.

\bigskip

\centerline{\includegraphics[scale=.7]{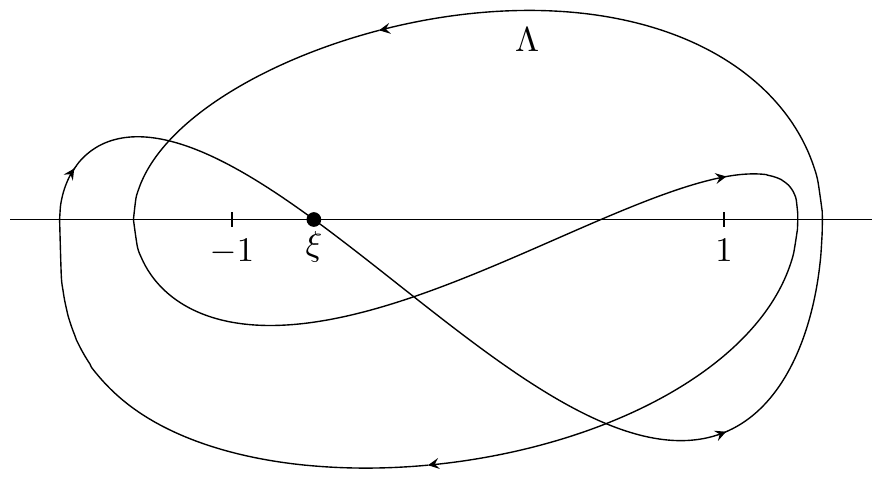}}
\centerline{Figure 2.1: Path $\Lambda$}

\bigskip

The point $\xi\in (-1,1)$ is the beginning and endpoint of
$\Lambda$. For $\alpha,\beta\in \CC$, we consider
$$
(1-z)^\alpha(1+z)^\beta=e^{\alpha\log(1-z)+\beta(1+z)}.
$$
It is a multi-valued function with branch points at $\infty$ and $\pm 1$. However, if we start with a value of $(1-z)^\alpha(1+z)^\beta$ at a particular value of $\Lambda$, and extend the definition of $(1-z)^\alpha(1+z)^\beta$ continuously along $\Lambda$, then we obtain a single-valued function on $\Lambda$ (if we view $\Lambda$ as a contour on the Riemann surface for the function $(1-z)^\alpha(1+z)^\beta$). For definiteness, we assume that the \textit{starting point} is $\xi\in (-1,1)$, and the branch of $(1-z)^\alpha(1+z)^\beta$ is such that $(1-\xi)^\alpha(1+\xi)^\beta>0$. In \cite{KAO}, it has been proved orthogonality for Jacobi polynomials using this contour. This orthogonality can be extended for exceptional Jacobi polynomials. Indeed, given $\alpha,\beta $ satisfying (\ref{cpar11}) and (\ref{cpar12}) ($\alpha$ and $\beta$ can be also complex numbers)
and a pair $\F=(F_1,F_2)$ of finite sets of positive integers, the complex orthogonality mentioned above for the exceptional Jacobi polynomials is the following:
there exists a contour $\Lambda$ as described above such that
\begin{align}\label{cojx}
   \int_{\Lambda} P_{n+u_\F}^{\alpha,\beta;\F}(z)P_{m+u_\F}^{\alpha,\beta;\F}(z)
   &\frac{(1-z)^{\alpha +k}(1+z)^{\beta+k_1-k_2}}{\Omega_\F^{\alpha,\beta}(z)^2} \, dz
   \\\nonumber
   &=-\frac{4e^{\pi i(\alpha+\beta)}\sin(\pi \alpha)\sin(\pi\beta)}{2^{-\alpha-\beta-1}n!}\HH^{\alpha,\beta}_\F(n),
\end{align}
for every $n,m\not \in F_1$, where the function $\HH^{\alpha,\beta}_\F$ is defined in (\ref{defadmj}).

This complex orthogonality can be proved as the Lemma 1.3 in \cite{dupe}, using the factorizations in Lemma \ref{lfel} for the exceptional Jacobi polynomials instead of those given in Lemma 2.2 in \cite{dupe} for the exceptional Laguerre polynomials. The proof of this Theorem is then similar to that of Theorem 1.2 in \cite{dupe}.
\end{proof}

\bigskip
Using the complex orthogonality (\ref{cojx}) one can easily
deduce the norm of the exceptional Jacobi polynomials.

\begin{corollary} With the same hypothesis as in Theorem \ref{th6.3}, we have
for $n\not \in F_1$ (that is, $n+u_\F\in \sigma_\F$)
\begin{align*}
\int _{-1}^1 (P_{n+u_\F}^{\alpha,\beta ;\F}(x))^2 \frac{(1-x)^{\alpha+k}(1+x)^{\beta+k_1-k_2}}{(\Omega_\F^{\alpha,\beta}(x))^2}dx&=\frac{2^{\alpha+\beta+1}
}{n!}\HH^{\alpha,\beta}_\F(n),
\end{align*}
where the function $\HH^{\alpha,\beta}_\F$ is defined in (\ref{defadmj}).
\end{corollary}

We guess that the converse of Theorem \ref{alv} is also true. But we have only been able to prove it under the additional technical assumption that $\alpha+\beta+s_\F+1>0$, where the nonnegative integer  $s_\F$ is defined in (\ref{defs0f}).

\begin{theorem}\label{convth} Given  real numbers $\alpha$ and $\beta$ satisfying (\ref{cpar11})  and (\ref{cpar12}) and a pair $\F$ of finite sets of positive integers, assume that $\alpha $, $\beta$ and $\F$ are Jacobi admissible. Then
\begin{enumerate}
\item $\alpha +k>-1$ and $\beta+k_1-k_2>-1$.
\item If in addition we assume $\alpha+\beta+s_\F+1>0$ then $\Omega_\F ^{\alpha,\beta}(x)\not =0$ for $x\in [-1,1]$.
\end{enumerate}
\end{theorem}

We first prove part 1 of Theorem \ref{convth}

\begin{proof}[Proof of part 1 of Theorem \ref{convth}]
Assume that the positive integer $N$ is big enough so that $F_1\subset \{ 0,1,\cdots, N\}$. Taking into account (\ref{cpar11}) and (\ref{cpar12}), we have that $\alpha, \beta$ and $N$ satisfy (\ref{cpar21}) and (\ref{cpar22}).
Consider the measure $\tau _N$ defined by
\begin{equation}\label{defmt}
\tau _N=\sum _{x=0}^{N-k_1}\frac{\binom{\alpha+k+x}{x}\binom{\beta+N-k_2-x}{N-k_1-x}(h_{u_\F}^{\alpha,\beta,N;\F}(x))^2}{\Omega _\F^{\alpha,\beta,N}(x)\Omega_\F^{\alpha,\beta,N}(x+1)}\delta {y_{N,x}},
\end{equation}
where
\begin{equation}\label{defya}
y_{N,x}=1-2x/N.
\end{equation}
Since $\alpha ,\beta$ and $\F$ are Jacobi admissible, by taking $N>\widehat{\alpha+\beta+1}$ we have that $\alpha,\beta$ and $N$ are also Hahn admissible (see Remark \ref{rm1}). Hence the measure $\tau_N$ is either positive or negative (Theorem \ref{th4.5}).

Consider the positive integer $s_\F$ and the pair $\F_\Downarrow$ defined by (\ref{defs0f}) and (\ref{deffd}), respectively.
We need the following limits
\begin{align}\label{lm1}
\lim _{N\to +\infty}\frac{\Omega _\F^{\alpha,\beta,N}((1-x)N/2)}{N^{u_\F+k_1}}&=\upsilon^\alpha_{\F}\Omega _\F^{\alpha,\beta} (x),\\\label{lm11}
\lim _{N\to +\infty}\frac{\Omega _\F^{\alpha,\beta,N}((1-x)N/2+1)}{N^{u_\F+k_1}}&=\upsilon^\alpha_{\F}\Omega _\F^{\alpha,\beta} (x),\\\label{lm2}
\lim _{N\to +\infty}\frac{h _{u_\F}^{\alpha,\beta,N;\F}((1-x)N/2)}{N^{u_\F}}&=\upsilon^{\alpha;\F}_{u_\F}P _{u_\F}^{\alpha,\beta;\F}(x),\\\label{lm3}
\lim _{N\to +\infty}\frac{\binom{\alpha+k+(1-x)N/2}{(1-x)N/2}\binom{\beta+N-k_2-(1-x)N/2}{N-k_1-(1-x)N/2}}{N^{\alpha+\beta+2k_1}}&=\frac{(1-x)^{\alpha+k}
(1+x)^{\beta+k_1-k_2}}{c},
\end{align}
uniformly in compact sets of the interval $(-1,1)$, where $\upsilon^{\alpha;{\F}}_{u_\F}$ and $\upsilon^\alpha_{\F}$ are defined by (\ref{defupn}) and (\ref{defup}), respectively, and $c$ is given by
\begin{equation}\label{defc}
c=2^{\alpha+\beta+2k_1}
\Gamma(\alpha+k+1)\Gamma(\beta+k_1-k_2+1).
\end{equation}
The  first limit is (\ref{lim2}). The second one is a consequence of
(\ref{lim3}). The third one is (\ref{lim1}). The forth one is consequence of the asymptotic behavior of $\Gamma(z+u)/\Gamma(z+v)$ when $z\to\infty$ (see \cite{EMOT}, vol. I (4), p. 47).

Lemma \ref{v0le} shows that $\Omega_\F^{\alpha,\beta} (\pm 1)\not =0$.  We take a real number $u$ with $0<u<1$ such that $\Omega _\F^{\alpha,\beta} (x) \not =0$ for $x\in [u,1]$. For a real number $v$ with $u<v<1$, write $I=[u,v]$, Then $\Omega _\F^{\alpha,\beta} $ does not vanish in $I$. Applying Hurwitz's Theorem to the limits (\ref{lm1}) and (\ref{lm11}) we can choice a contable set $X$ of positive integers with $\lim_{N\in X} N=+\infty$ such that $\Omega _\F^{\alpha,\beta,N}((1-x)N/2)\Omega _\F^{\alpha,\beta,N}((1-x)N/2+1)\not =0$, $x\in I$ and $N\in X$.

Hence, using (\ref{defupn}) and (\ref{defup}), we can combine the limits (\ref{lm1}), (\ref{lm11}), (\ref{lm2}) and (\ref{lm3}) to get
\begin{equation}\label{lm4}
\lim _{N\to +\infty;N\in X}H_N(x)=\frac{4^{k_1}}{c}H(x),\quad \mbox{uniformly in $I$, where}
\end{equation}

\begin{align*}
H_N(x)&=\frac{\binom{\alpha+k+(1-x)N/2}{(1-x)N/2}\binom{\beta+N-k_2-(1-x)N/2}{N-k_1-(1-x)N/2}(h_{u_\F}^{\alpha,\beta,N})^2((1-x)N/2)}{N^{\alpha+\beta}\Omega _\F^{\alpha,\beta,N}((1-x)N/2)\Omega _\F^{\alpha,\beta,N}((1-x)N/2+1)},\\
H(x)&=\frac{(1-x)^{\alpha+k}(1+x)^{\beta+k_1-k_2}(P_{u_\F}^{\alpha,\beta,N})^2(x)}
{(\Omega_\F^{\alpha,\beta})^2(x)}.
\end{align*}

We now prove that
\begin{equation}\label{lm5}
\lim _{N\to +\infty ;N\in X}\frac{2\tau _N(I)}{N^{\alpha+\beta+1}}=\frac{4^{k_1}}{c}\int_{I}H(x)dx.
\end{equation}
To do that, write $I_N=\{ x\in \NN: (1-v)N/2\le x\le (1-u)N/2\}$, ordered in decreasing size.
The numbers $y_{N,x}$, $x\in I_N$, form a partition of the interval $I$ with $y_{N,x+1}-y_{N,x}=2/N$ (see (\ref{defya})). Since the function $H$ is continuous  in the interval $I$, we get that
$$
\int_{I}H(x)dx=\lim_{N\to 1; N\in X}S_N,
$$
where $S_N$ is the Cauchy sum
$$
S_N=\sum_{x\in I_N}H(y_{N,x})(y_{N,x+1}-y_{N,x}).
$$
On the other hand, since $x\in I_N$ if and only if $u\le y_{N,x}\le v$ (\ref{defya}), we get
\begin{align*}
\frac{2\tau _N(I)}{N^{\alpha+\beta+1}}&=\frac{2}{N^{\alpha+\beta+1}}\sum _{x\in I_N}\frac{\binom{\alpha+k+x}{x}\binom{\beta+N-k_2-x}{N-k_1-x}(h_{u_\F}^{\alpha,\beta,N})^2(x)}{\Omega _\F^{\alpha,\beta,N}(x)\Omega _\F^{\alpha,\beta,N}(x)}\\
&=\frac{2}{N}\sum _{x\in I_N}H_N(y_{N,x})=\sum _{x\in I_N}H_N(y_{N,x})(y_{N,x+1}-y_{N,x}).
\end{align*}
The limit (\ref{lm5}) now follows from the uniform limit (\ref{lm4}).

Since $H(x)>0$, $x\in I$, (\ref{lm5}) gives that $c$ has the same sign as the measure $\tau _N$ (let us remind that $\tau _N$ is either a positive or a negative measure).

The identity (\ref{pf5}) and (\ref{mraf}) says that
$$
\tau _N(\RR)=\frac{\prod_{f\in F_1}(\lambda(0)-\lambda(f))\prod_{f\in F_2}(\lambda(0)-\lambda(f-\beta))}{(\alpha+1)_k(\beta+1)_{k_1-k_2}w_{*,\alpha,\beta,N}(0)}.
$$
Using (\ref{masdh}), one gets
\begin{equation}\label{lmn1}
\lim_{N\to+\infty}\frac{2\tau _N(\RR)}{N^{\alpha+\beta+1}}=d,
\end{equation}
where  $d$ is certain non-null real number which has the same sign as the measure $\tau _N$. Hence $c$ and $d$ have the same sign.
We also have
\begin{equation}\label{lmn2}
\tau_N(I)\begin{cases}\le \tau_N(\RR),& \mbox{if $\tau_N$ is a positive measure},\\
\ge \tau_N(\RR),& \mbox{if $\tau_N$ is a negative measure}.
\end{cases}
\end{equation}
Hence, using (\ref{lm5}), (\ref{lmn1}), (\ref{lmn2}) and taking into account that $\tau_N$, $c$ and $d$ has the same sign,
one gets
$$
\int_{I}H(x)dx \le \frac{cd}{4^{k_1}}.
$$
That is
$$
\int _u^v\frac{(1-x)^{\alpha +k}(1+x)^{\beta+k_1-k_2}(\Omega_{\F_{\Downarrow }}^{\alpha +s_\F,\beta +s_\F})^2(x)}{(\Omega_\F^{\alpha,\beta} )^2(x)}dx\le \frac{cd}{4^{k_1}z^2_{\alpha,\beta;\F}},
$$
where we have use (\ref{rromh}).

Notice that  since $\alpha, \beta$ and $\F$ satisfy (\ref{cpar11}) and (\ref{cpar12}),
$\alpha+s_F$, $\beta+s_F$ and $\F_{\Downarrow }$  satisfy (\ref{cpar11}) and (\ref{cpar12}) as well. Using
Lemma \ref{v0le}, we get $\Omega_{\F_{\Downarrow }}^{\alpha +s_\F,\beta +s_\F}(1)\not =0$. Then,
if $\alpha +k\le -1$, we would deduce
$$
\lim_{v\to 1^-}\int _u^v\frac{(1-x)^{\alpha +k}(1+x)^{\beta+k_1-k_2}(\Omega_{\F_{\Downarrow }}^{\alpha +s_\F,\beta +s_\F})^2(x)}{(\Omega_\F^{\alpha,\beta} )^2(x)}dx=+\infty,
$$
which it is a contradiction. Hence $\alpha +k> -1$.

Proceeding in the same way but using $x=-1$ instead of $x=1$, we can prove that $\beta+k_1-k_2> -1$.
\end{proof}

Notice that part 1 of Theorem \ref{convth} can be proved under the weaker assumption of being $\alpha,\beta$ and $N$  Hahn admissible for $N>\widehat{\alpha+\beta+1}$ (according to part 3 of Lemma \ref{l3.1}, this implies the constant sign of the measure $\tau_N$ (\ref{defmt})). So part 1 of Theorem \ref{convth} actually provides a proof for part 1 of Lemma \ref{ladm}. We are now ready to prove part 2 of Lemma \ref{ladm} which  also  follows from part 1 of Theorem \ref{convth}.

\begin{proof}[Proof of part 2 of Lemma \ref{ladm}] Since we have already proved that $\alpha+k+1>0$ and $\beta+k_1-k_2+1>0$, we deduce that the constant $c$
(\ref{defc}) is positive. The limit (\ref{lm5}) then shows that the measure $\tau_N$ (\ref{defmt}) is also positive for $N$ big enough. That is, for $x=0,\cdots, N-k_1$,
$$
\frac{\binom{\alpha+k+x}{x}\binom{\beta+N-k_2-x}{N-k_1-x}}{\Omega _\F^{\alpha,\beta,N}(x)\Omega_\F^{\alpha,\beta,N}(x+1)}>0.
$$
Using now part ii of Lemma \ref{l3.1}, we deduce that for $N$ big enough the sign of $\HH_\F^{\alpha,\beta,N}$ (\ref{defadmh}) is equal to the sign of $(\alpha+1)_k(\beta+1)_{k_1-k_2}$. It is now enough to take into account that for $N>\widehat{\alpha+\beta+1}$, the sign of $\HH_\F^{\alpha,\beta,N}$ does not depend on $N$ (\ref{ndN}).
\end{proof}

To prove part 2 of Theorem \ref{convth} we need the following Lemma.

\begin{lemma}\label{lem6.2} With the same hypothesis of Theorem \ref{convth}, assume
in addition that $\Omega_\F ^{\alpha,\beta}(x_0)=0$ for some $x_0\in (-1,1)$, then $\Omega_{\F_{\Downarrow }}^{\alpha +s_\F,\beta +s_\F}(x_0)=0$ as well.
\end{lemma}

\begin{proof}
We take a real number $v$ with $x_0<v<1$ such that $\Omega _\F^{\alpha,\beta} (x) \not =0$ for $x\in (x_0,v]$. For a real number $u$ with $x_0<u<v$, write $I=[u,v]$. Then $\Omega _\F^{\alpha,\beta} $ does not vanish in $I$. Proceeding as in the
proof of part 1 of Theorem \ref{convth}, we get
$$
\int _u^v\frac{(1-x)^{\alpha +k}(1+x)^{\beta+k_1-k_2}(\Omega_{\F_{\Downarrow }}^{\alpha +s_\F,\beta +s_\F})^2(x)}{(\Omega_\F^{\alpha,\beta} )^2(x)}dx\le \frac{cd}{4^{k_1}}.
$$
If $\Omega_{\F_{\Downarrow }}^{\alpha +s_\F,\beta +s_\F}(x_0)\not =0$, since $\Omega_\F^{\alpha,\beta} (x_0)=0$ we get
$$
\lim_{u\to x_0^+}\int _u^v\frac{(1-x)^{\alpha +k}(1+x)^{\beta+k_1-k_2}(\Omega_{\F_{\Downarrow }}^{\alpha +s_\F,\beta +s_\F})^2(x)}{(\Omega_\F^{\alpha,\beta} )^2(x)}dx=+\infty.
$$
Hence $\Omega_{\F_{\Downarrow }}^{\alpha +s_\F,\beta +s_\F}(x_0)=0$.
\end{proof}

If $F_1=\emptyset$, the converse of Theorem \ref{alv} is true.

\begin{theorem} Given  real numbers $\alpha$ and $\beta$ satisfying (\ref{cpar11})  and (\ref{cpar12}) and a pair $\F=(\emptyset, F_2)$, assume that $\alpha $, $\beta$ and $\F$ are admissible. Then (\ref{superc}) holds.
\end{theorem}

\begin{proof}
We only have to prove that $\Omega _\F^{\alpha,\beta} (x) \not =0$, $-1<x<1$.
We prove it by induction on $k_2$. For $k_2=1$, we have that $F_2$ is a singleton $F_2=\{f\}$, and then $\Omega _\F^{\alpha,\beta} (x)=P^{\alpha,-\beta}_f(x)$. Since $k_1=0$, $k_2=1$, we have from the first part of Lemma \ref{lem6.2} that $\alpha>-2$ and $\beta>0$. Hence, using the admissibility condition (\ref{defadmj}), we deduce that either $-1<\alpha$ and $f<\beta$ or $-2<\alpha<-1$ and $f-1<\beta<f$. In both cases, according to Theorem 6.72 in \cite{Sz} the polynomial $P^{\alpha,-\beta}_f(x)$ does not vanish in $(-1,1)$.

Assume now that the theorem holds for $k_2\le s$, and take a finite set of positive integers $F_2$, with $k_2=s+1$ elements.
According to the definition of $s_{F_1}$ (\ref{defs0}) for $F_1=\emptyset,$ we have $s_{F_1}=1$. Hence we also have $s_\F=1$ (see (\ref{defs0f})). If
there exists $-1<x_0<1$ such that $\Omega _\F^{\alpha,\beta} (x_0)=0$, using the previous Lemma, we get that also $\Omega _{\F _{\Downarrow}}^{\alpha +1,\beta+1} (x_0)=0$. Since $F_1=\emptyset$, we have $\F=\F _{\Downarrow}$ (see (\ref{deffd})) and  then $\Omega _{\F }^{\alpha +1,\beta+1} (x_0)=0$.
If $\alpha, \beta$ and $\F$ are admissible with $F_1=\emptyset$, then $\alpha+1, \beta+1$ and $\F$ are also admissible (see part 3 of Lemma \ref{ladm}).
Proceeding as before, we can conclude that $\Omega _\F^{\alpha +j,\beta+j} (x_0)=0$, $j=0,1,2,\ldots $
Consider the pair $\F_{2,\{ s+1\} }$ defined by (\ref{deff2}). Since $F_1=\emptyset$, from the definition of Jacobi admissibility (\ref{defadmj}), we deduce that there exists $h_0\in \NN$ such that for $h\ge h_0$, $\alpha+h,\beta+h$ and $\F_{2,\{ s+1\} }$ are admissible. Write $\tilde \alpha=\alpha+h_0$, $\tilde \beta=\beta+h_0$. We also have $\Omega _\F^{\tilde \alpha +j,\tilde \beta+j} (x_0)=0$, $j=0,1,2,\ldots $

For a positive integer $m\ge s+\max F_2+1> s+1$ consider the $(s+1)\times m$ matrix
$$
M=\left(
  \begin{array}{@{}c@{}cccc@{}c@{}}
   &  &&\hspace{-.9cm}{}_{1\le r\le m} \\
    \dosfilas{ (\tilde\beta-f)_{r-1}(1+x_0)^{s+1-r}P_{f}^{\tilde \alpha+r-1,-\tilde \beta-r+1}(x_0) }{f\in F_2}
  \end{array}
  \right) .
$$
Write $c_i$, $i=1,\ldots, m$, for the columns of $M$ (from left to right).
For $j\ge 0$, consider the $(s+1)\times s$ submatrix $M_j$ of $M$ formed by the consecutive columns $c_{j+i}$, $i=1,\cdots s$, of $M$.
Using (\ref{defhom}), we see that the minor of $M_j$ formed by its first $s$ rows is equal to $(1+x_0)^{s(1-j)}\prod_{f\in F_2\setminus {\{ f^{2\rceil}_{s+1}\}}}(\tilde\beta -f)_j\Omega _{\F_{2,\{ s+1\} }}^{\tilde \alpha +j,\tilde \beta+j}(x_0)$ where the pair $\F_{2,\{ s+1\} }$ is defined by (\ref{deff2}).
Since the set $F_2\setminus {\{ f^{2\rceil}_{s+1}\}}$ has $s$ elements and $\tilde\alpha +j$, $\tilde\beta +j$ and $\F_{2,\{ s+1\} }$ are admissible, the induction hypothesis says that $\Omega _{\F_{2,\{ s+1\} }}^{\tilde\alpha +j,\tilde\beta +j}(x_0)\not =0$, and hence
the columns $c_{j+i}$, $i=1,\cdots s$, of $M$ are linearly independent.
On the other hand, the consecutive columns $c_{j+i}$, $i=1,\cdots s+1$, of $M$ are linearly dependent because its determinant is equal to  $(1+x_0)^{(s+1)(1-j)}\prod_{f\in F_2}(\tilde\beta -f)_j\Omega _{\F }^{\tilde\alpha +j,\tilde\beta +j}(x_0)$
and $\Omega _{\F }^{\tilde\alpha +j,\tilde\beta +j}(x_0)=0$. Using Lemma \ref{rmc}, we conclude that $\rank M=s$.
Write now
\begin{equation}\label{defhomf2v}
\tilde M=\left(
  \begin{array}{@{}c@{}cccc@{}c@{}}
  &  &&\hspace{-.9cm}{}_{1\le r\le m} \\
     \dosfilas{((1+x)^{s}P_{f}^{\tilde\alpha,-\tilde\beta}(x))^{(r-1)}_{\vert x=x_0} }{f\in F_2}
 \end{array}
  \hspace{-.4cm}\right).
\end{equation}
Using (\ref{Lagab}), it is easy to see that $\rank \tilde M=\rank M=s$. Then there exist numbers $e_f$, $f\in F_2$, not all zero such that the  polinomial $p(x)=\sum_{f\in F_2}e_f(1+x)^{s}P_{f}^{\tilde\alpha,-\tilde\beta}(x)$ is non null and has a zero of multiplicity $m$ in $x_0$. But  the polynomial $p$ has degree at most $s+\max F_2$, and since $m\ge s+\max F_2+1>\deg p$, this shows that $p=0$, which it is a contradiction. This proves the theorem.
\end{proof}

We finally prove part 2 of Theorem \ref{convth}

\begin{proof}
Write $s=\max F_1$. We proceed by complete induction on $s$. The case $s=-1$ (i.e., $F_1=\emptyset$) is just the previous Theorem (which it has been proved without using the hypothesis $\alpha+\beta+s_\F+1>0$).

Assume now that $\alpha$, $\beta$ and $\F$ are Jacobi admissible, $\alpha+\beta+s_\F+1>0$ and
\begin{equation}\label{pas3}
\Omega _\F^{\alpha,\beta} (x) \not =0,\quad -1<x<1,
\end{equation}
holds for $\max F_1\le s$.

We now prove that if $\max F_1=s+1$, and $\alpha$, $\beta$ and $\F$ are Jacobi admissible with $\alpha+\beta+s_\F+1>0$
then (\ref{pas3}) also holds.

Consider the pair $\F _{\Downarrow}=\{(F_1)_{\Downarrow},F_2\}$ defined by (\ref{deffd}).
Since $F_1\not =\emptyset$, we have that $\max (F_1)_{\Downarrow}\le s $. The part 4 of Lemma \ref{ladm} says that if
$\alpha +s_\F$, $\beta+s_\F$ and $\F _{\Downarrow}$ are Jacobi admissible as well. If we write $\tilde \alpha=\alpha +s_F$, $\tilde \beta=\beta+s_F$, we also have $\tilde \alpha+\tilde \beta+s_{\F_\Downarrow}+1=(\alpha+\beta+s_\F+1)+s_\F+s_{\F_\Downarrow}>0$.

The induction hypothesis (\ref{pas3}) then says that
$\Omega _{\F _{\Downarrow}}^{\alpha +s_\F, \beta+s_\F} (x) \not =0$ for $-1<x<1$. The second part of Lemma \ref{lem6.2} then gives that also $\Omega _\F^{\alpha,\beta} (x) \not =0$, for $-1<x<1$.
\end{proof}

We need the technical assumption $\alpha+\beta+s_\F+1>0$ in the previous proof because otherwise we can not guarantee the
admissibility of $\alpha +s_\F$, $\beta+s_\F$ and $\F _{\Downarrow}$ from the admissibility of $\alpha $, $\beta$ and $\F$, as the following example shows. Consider $\alpha=-3/2$, $\beta=-9/7$, $F_1=\{2,3,4\}$ and $F_2=\{1,2\}$, for which $s_\F=1$ and
$(F_1)_{\Downarrow}=\{1,2,3\}$. It is then easy to see that $\alpha,\beta$ and $\F=(F_1,F_2)$ are admissible but
$\alpha+s_\F,\beta+s_\F$ and $\F _{\Downarrow}=((F_1) _{\Downarrow},F_2)$ is not (because $\HH^{\alpha+s_\F,\beta+s_\F}_{\F _{\Downarrow}}(0)$ is negative). But although $\alpha$, $\beta$ and $\F$ do not satisfy the assumption $\alpha+\beta+s_\F+1>0$, they are not a counterexample for the converse of Theorem \ref{alv} because it is easy to check that $\Omega_\F^{\alpha,\beta}$ does not vanish in $[-1,1]$.

\section{Recurrence relations}
The gaps in their degrees imply that exceptional Hahn or Jacobi orthogonal polynomials do not satisfy three term recurrence relations as the usual orthogonal polynomials do. However, as happens with exceptional Charlier, Meixner, Hermite or Laguerre polynomials, they also satisfy higher order recurrence relations of the form
\begin{equation}\label{horr}
\Upsilon (x) p_n(x)=\sum_{j=-r}^ra_{n,j}p_{n+j},\quad n\ge n_0,
\end{equation}
where $\Upsilon$ is a polynomial of degree $r$, $(a_{n,j})_{n}$, $j=-r,\cdots ,r$, are sequences of numbers independent of $x$ (called recurrence coefficients), with $a_{n,r}\not =0$, for $n$ big enough and $n_0$ is certain nonnegative integer. We say that this high order recurrence relation has order $2r+1$. As shown in \cite{durr} this is a consequence of the duality of the exceptional discrete and Krall discrete polynomials.

\subsection{Exceptional Hahn polynomials}
Our procedure to construct higher order recurrence relations for the exceptional Hahn polynomials consists in applying duality to the higher order difference operator with respect to which the associated Krall discrete polynomials (see (\ref{defqnme})) are eigenfunctions. Since we want to work with orthogonal polynomials with respect to positive measures we assume that $N$ is a positive integer, $\alpha, \beta $ and $N$ satisfy (\ref{cpar21}) and (\ref{cpar22}) and that $\alpha, \beta $, $N$ and $\F$ are Hahn admissible, although these assumptions are not needed for the implementation of our method to find higher order recurrence relations for the polynomials (\ref{defmex}).

Up to an additive constant, we define the polynomial $\Upsilon _\F^{\alpha,\beta,N}$  of degree $w_F$ (see (\ref{defwf2})) by solving the
first order difference equation
\begin{equation}\label{lch}
\Upsilon _\F^{\alpha,\beta,N}(x)-\Upsilon _\F^{\alpha,\beta,N}(x-1)=\Omega _\G^{-\alpha+\max F_1+\max F_2+2,-\beta+\max F_1-\max F_2,-N-3-\max F_1}(-x),
\end{equation}
with $\G=(I(F_1),I(F_2))$ and $I$ the involution defined by (\ref{dinv}).
In \cite{dudh}, Corollary 5.2, it is proved that the polynomials $q_n^{\F}(\lambda(x+u_\F))$, $n\ge 0$, (see (\ref{defqnme})) are eigenfunctions of a higher order difference operator $D_\F$ (which can be explicitly constructed using \cite{dudh}, Theorem 3.1). The associated eigenvalues are given
by the polynomial $\Upsilon _\F^{\alpha,\beta,N}$, so that $D_\F(q_n^{\F}(\lambda(x)))=\Upsilon _\F^{\alpha,\beta,N}(n)q_n^{\F}(\lambda(x))$.

In \cite{dudh}, Corollary 5.2, it is also proved that $D_\F$ is a difference operator of order $2w_F+1$, which can be written in terms of the shift operators $\Sh_j$, $\Sh_j(f)=f(x+j)$, in the form
\begin{equation}\label{expexp2}
D_\F=\sum_{j=-w_\F}^{w_\F} g_j(x)\Sh _j,
\end{equation}
where $g_j$, $j=-w_F,\cdots, w_F$, are certain rational functions.

The duality (\ref{duaqnrn}) gives then the higher order recurrence relation for the polynomials $h_n^{\alpha,\beta,N;\F}$, $n\sigma_\F$ (the proof is similar to Corollary 2.3 in \cite{durr} and it is omitted).

\begin{corollary}\label{cor1} The exceptional Hahn polynomials  satisfy a $2w_F+1$ order recurrence relation of the form
\begin{equation}\label{horrchex}
\sum_{j=-w_\F}^{w_\F}A_j^{\alpha,\beta,N;\F}(n)h_{n+j}^{\alpha,\beta,N;\F}(x)=\Upsilon _\F^{\alpha,\beta,N} (x)c_nh_{n}^{\alpha,\beta,N;\F}(x),\quad n\ge 0,
\end{equation}
where $c_n$ is a normalization constant and the number $w_\F$ and the polynomial $\Upsilon_\F^{\alpha,\beta,N}$ are defined by (\ref{defwf2}) and (\ref{lch}), respectively. For $j=-w_\F,\cdots , w_\F$, $A_j^{a;\F}(n)$ is a rational function in $n$ which does not depend on $x$ (and whose denominator does not vanish for $n\in \NN$).
\end{corollary}

The expression we have found in \cite{dudh} for the higher order difference operator $D_\F$ makes difficult to find explicitly the coefficients of its expansion (\ref{expexp2}) in terms of the shift operators $\Sh_n$, $n\in \ZZ$. These coefficients are needed to find explicit expressions for the recurrence coefficients $(A_j)_{j=-r}^r$ in (\ref{horrchex}). We have not been able to find explicit formulas for them in terms of arbitraries $\alpha$, $\beta$, $N$ and $\F$, but such explicit formulas
can be found for small values of $w_F$. Here it is an example.
Consider $F_1=\emptyset$, $F_2=\{1\}$. It is easy to see that for $N\ge 3$, $\alpha$, $\beta$, $N$ and $\F$ are Hahn admissible if and only if either $-1<\alpha$ and $1<\beta $ or $-2<\alpha<-1$ and $0<\beta<1$.

We have $w_F=2$ and then according to the Corollary \ref{cor1}, the polynomials $(h_n^{\alpha,\beta,N;\F})_n$ satisfy a five term recurrence relation. Proceeding as in \cite{durr} we can explicitly find the coefficients $A_j^{\alpha,\beta,N;\F}$, $j=-2,\cdots , 2$:
\begin{equation}\label{coefm1}
A_j^{\alpha,\beta,N;\F}(n)=\begin{cases}
\frac{3(\beta-\alpha-2)\binom{n}{3}(\alpha+n+1)(\beta+n-2)_2(N-n+2)_2(\alpha+\beta+N+n-1)_2}
{(\alpha+1)(\alpha+n-1)(\alpha+\beta+2n-4)_4}, &\mbox{if $j=-2$},\\
\frac{2(\alpha+\beta)(\alpha+\beta+2N+2)\binom{n}{2}(\alpha+n+1)(N-n+2)(\alpha+\beta+N+n)(\beta+n-2)_2}
{(\alpha+1)(\alpha+\beta+2n-4)(\alpha+\beta+2n-2)_3}, &\mbox{if $j=-1$},\\
\frac{-(\alpha+n-1)(\beta+n-4)A_{-2}^{\alpha,\beta,N;\F}(n)}{(\alpha+n+1)(\beta+n-2)(N-n+2)_2}
+\frac{(\alpha+n)(\beta+n-3)A_{-1}^{\alpha,\beta,N;\F}(n)}{(\alpha+n+1)(\beta+n-2)(N-n+2)}
&\\\qquad+\frac{(\alpha+n+2)(\beta+n-1)(N-n+1)A_{1}^{\alpha,\beta,N;\F}(n)}{(\alpha+n+1)(\beta+n-2)}
&\\\qquad -\frac{(\alpha+n+3)(\beta+n)(N-n)_2 A_{2}^{\alpha,\beta,N;\F}(n)}{(\alpha+n+1)(\beta+n-2)}, &\mbox{if $j=0$},\\
\frac{(\alpha+\beta)(\alpha+\beta+2N+2)n(\beta+n-2)(\alpha+\beta+n)(\alpha+n)_2}
{(\alpha+1)(\alpha+\beta+2n+2)(\alpha+\beta+2n-2)_3}, &\mbox{if $j=1$},\\
\frac{(\beta-\alpha-2)n(\beta+n-2)(\alpha+n)_2(\alpha+\beta+n)_2}
{2(\alpha+1)(\beta+n)(\alpha+\beta+2n-1)_4}, &\mbox{if $j=2$}
\end{cases}
\end{equation}
and $\displaystyle \Upsilon_\F^{\alpha,\beta,N}(x)=\frac{(\beta-\alpha-2)x^2}{2(\alpha+1)}+
\frac{(2\beta\alpha+3\beta-\alpha+2N-2+2N\alpha)x}{2(\alpha+1)}$, $c_n=n$.

\bigskip
When $N$ is a positive integer and $F_1\subset \{0,1,2,\cdots, N\}$, with $\max F_1>N/2$, the order $2w_\F+1$ in the recurrence formula (\ref{horrchex}) can be improved. Indeed,  consider the measure
\begin{equation}\label{mii}
\rho _{\alpha,\beta,N}^{\G}=\prod_{f\in G_1}(\lambda-\lambda(N-f))\prod_{f\in G_2}(\lambda-\lambda(f))\prod_{f\in G_3}(\lambda-\lambda(f-\beta))w _{*,\alpha,\beta,N},
\end{equation}
where $\G=(G_1,G_2,G_3)$ is a trio of finite set of positive integers. Notice that for $G_1=\emptyset$, $G_2=F1$ and $G_3=F_2$, we get the measure (\ref{mraf}). Using appropriate representations of the measure (\ref{mraf}) in the form
(\ref{mii}), it is proved in \cite{dudh} that (under mild conditions on the parameters) the orthogonal polynomials with respect to the measure $\rho _{\alpha,\beta,N}^{\G}$ are eigenfunctions of a higher order difference operator of the form (\ref{expexp2}) with
$$
-s=r=\sum_{f\in F_1;f\le N/2}f+\sum_{f\in F_1;f>N/2}(N-f)-\binom{m_1}{2}-\binom{m_2}{2}-\binom{k_2}{2}.
$$
where $m_1$ and $m_2$ are the number of elements of $G_1=\{N-f:f\in F_1,f>N/2\}$ and $G_2=\{f:f\in F_1,f\le N/2\}$, respectively. Notice that this number is less than or equal to $w_\F$ (\ref{defwf2}) (but equal to $w_\F$ for $N$ big enough).

\subsection{Exceptional Jacobi polynomials}
Since we want to work with orthogonal polynomials with respect to positive measures we assume that $\alpha, \beta $ satisfy (\ref{cpar11}) and (\ref{cpar12}) and that (\ref{superc}) holds (a sufficient condition for this last assumption is given in Theorem \ref{convth}), although these assumptions are not needed for the implementation of our method to find higher order recurrence relations for the polynomials (\ref{deflax}).
Up to an additive constant, we define the polynomial $\Upsilon_{\alpha,\beta;\F}$ of degree $w_\F$ as the solution of the first order differential equation
\begin{equation}\label{mm2}
\Upsilon _{\alpha,\beta;\F}'(x)=\Omega _\G^{-\alpha+\max F_1+\max F_2+2,-\beta+\max F_1-\max F_2}(x),
\end{equation}
where $\G=(I(F_1),I(F_2))$, $I$ is the involution defined by (\ref{dinv}) and $\Omega_\G^{\alpha,\beta} $ is the Wronskian type determinant (\ref{defhom}).

By taking limit in Corollary \ref{cor1} we get

\begin{corollary}\label{cor4} The exceptional Jacobi polynomials satisfy a $2w_\F+1$ order recurrence relation of the form
\begin{equation}\label{horrlax}
\sum_{j=-w_\F}^{w_\F}A_j^{\alpha,\beta;\F}(n)P_{n+j}^{\alpha,\beta;\F}(x)=\Upsilon_{\alpha,\beta; \F} (x) P_{n}^{\alpha,\beta;\F}(x),
\end{equation}
where the number $w_\F$ and the polynomial $\Upsilon_{\alpha,\beta ;\F}$ are defined by (\ref{defwf2}) and (\ref{mm2}), respectively. For $j=-w_\F,\cdots , w_\F$, $A_j^{\alpha,\beta;\F}(n)$ is a rational function in $n$ which does not depend on $x$ and whose denominator does not vanish for $n\in \NN$.
\end{corollary}

\bigskip
Consider $F_1=\emptyset$, $F_2=\{1\}$. Since $\alpha>-2$ and $\beta>0$, and $s_\F=1$, we get that $\Omega_\F^{\alpha,\beta}(x)\not=0$ for $x\in [-1,1]$ is equivalent to the admissibility of $\alpha,\beta$ and $\F$. And it is easy to see that
this happens if and only if either $-1<\alpha$ and $1<\beta $ or $-2<\alpha<-1$ and $0<\beta<1$.
By taking limit in (\ref{coefm1}), we get that the exceptional Hahn polynomials $(P_n^{\alpha,\beta;\F})_n$ satisfy the five order recurrence relation (\ref{horrlax}) where $w_F=2$,
\begin{equation*}\label{coefm2}
A_j^{\alpha,\beta;\F}(n)=\begin{cases}
\frac{(\beta-\alpha-2)(\alpha+n+1)(\alpha+n-2)(\beta+n-2)_2}
{2(\alpha+1)(\alpha+\beta+2n-4)_4}, &\mbox{if $j=-2$},\\
\frac{-2(\alpha+\beta)(\alpha+n-1)(\alpha+n+1)(\beta+n-2)_2}
{(\alpha+1)(\alpha+\beta+2n-4)(\alpha+\beta+2n-2)_3}, &\mbox{if $j=-1$},\\
-\frac{(n-1)(n-2)(\beta+n-4)A_{-2}^{\alpha,\beta;\F}(n)}{(\alpha+n-2)(\alpha+n+1)(\beta+n-2)}
-\frac{(n-1)(\alpha+n)(\beta+n-3)A_{-1}^{\alpha,\beta;\F}(n)}{(\alpha+n-1)(\alpha+n+1)(\beta+n-2)}
&\\\qquad-\frac{(\alpha+n+2)(\alpha+n)(\beta+n-1)A_{1}^{\alpha,\beta;\F}(n)}{n(\alpha+n+1)(\beta+n-2)}
&\\\qquad-\frac{(\alpha+n+3)(\alpha+n)(\beta+n)A_{2}^{\alpha,\beta;\F}(n)}{n(n+1)(\beta+n-2)}, &\mbox{if $j=0$},\\
\frac{-2(\alpha+\beta)n(\alpha+n+1)(\beta+n-2)(\alpha+\beta+n)}
{(\alpha+1)(\alpha+\beta+2n+2)(\alpha+\beta+2n-2)_3}, &\mbox{if $j=1$},\\
\frac{(\beta-\alpha-2)\binom{n+1}{2}(\beta+n-2)(\alpha+\beta+n)_2}
{(\alpha+1)(\beta+n)(\alpha+\beta+2n-1)_4}, &\mbox{if $j=2$}
\end{cases}
\end{equation*}
and $\displaystyle \Upsilon_\F^{\alpha,\beta}(x)=\frac{(1-x)(2+3\alpha+\beta+(\alpha-\beta+2))x}{8(\alpha+1)}$.

\bigskip

\noindent
\textit{Mathematics Subject Classification: 42C05, 33C45, 33E30}

\noindent
\textit{Key words and phrases}: Orthogonal polynomials. Exceptional polynomial. Difference and differential operators.
Hahn polynomials. Jacobi polynomials.

     \end{document}